\documentclass[11pt,letterpaper]{amsart}
\usepackage{amssymb, latexsym, graphicx, mathrsfs, enumerate, amsmath, amsthm, textcomp, url, tikz-cd, bbm, mathtools, multirow, multicol, bbm, etoolbox}
\usepackage[T1]{fontenc} 
\usepackage{mathdots}
\usepackage{comment}
\usepackage[toc,page]{appendix}
\usepackage[breaklinks=true]{hyperref}
\hypersetup{
    colorlinks=true,
    linkcolor=blue,
    filecolor=magenta,      
    urlcolor=cyan,
}

\tikzset{
    labl/.style={anchor=south, rotate=90, inner sep=.5mm}
}

\usepackage{color}

\input xy
\xyoption{all}

\allowdisplaybreaks
\allowdisplaybreaks

\numberwithin{equation}{section}
\newtheorem{theorem}{Theorem}[section]
\newtheorem{proposition}[theorem]{Proposition}

\newtheorem{lemma}[theorem]{Lemma}
\newtheorem{conjecture}[theorem]{Conjecture}

\newtheorem*{remark*}{Remark}

\theoremstyle{definition}
\newtheorem{definition}[theorem]{Definition}
\newtheorem{example}[theorem]{Example}
\newtheorem{remark}[theorem]{Remark}

\newcommand{\Q}{\mathbb{Q}}

\newcommand{\Z}{\mathbb{Z}}

\newcommand{\Mfo}{\mathcal{O}}

\makeatletter
\def\ops@declare#1{\expandafter\DeclareMathOperator\csname #1\endcsname{#1}}
\def\ops@scan#1,{\ifx#1\relax\let\ops@next\relax\else\ops@declare{#1}\let\ops@next\ops@scan\fi\ops@next}
\newcommand{\DeclareMathOperators}[1]{\ops@scan#1,\relax,}
\makeatother
\DeclareMathOperators{Hom,Nat,Tor,Ann,Aff,Aut,GL,SL,SO,Spec,lt,lp,lc,im,lcm,id,Set,Ab,Grp,Ring,Mod,Top,Vect,Ind,res,Gal,coker,sh,Supp, diag,Stab,inv,Nm,ord,End}
\def\makebb#1{\expandafter\def\csname b#1\endcsname{{\mathbb{#1}}}\ignorespaces}
\def\makebf#1{\expandafter\def\csname bf#1\endcsname{{\mathbf{#1}}}\ignorespaces}
\def\makegr#1{\expandafter\def\csname f#1\endcsname{{\mathfrak{#1}}}\ignorespaces}
\def\makescr#1{\expandafter\def\csname s#1\endcsname{{\mathscr{#1}}}\ignorespaces}
\def\makec#1{\expandafter\def\csname c#1\endcsname{{\mathcal{#1}}}\ignorespaces}
\def\makecal#1{\expandafter\def\csname cal#1\endcsname{{\mathcal{#1}}}\ignorespaces}
\def\doLetters#1{#1A #1B #1C #1D #1E #1F #1G #1H #1I #1J #1K #1L #1M
	#1N #1O #1P #1Q #1R #1S #1T #1U #1V #1W #1X #1Y #1Z}
\def\doletters#1{#1a #1b #1c #1d #1e #1f #1g #1h #1j #1k #1l #1m
	#1n #1o #1p #1q #1r #1s #1t #1u #1v #1w #1x #1y #1z}
\doLetters\makebb
\doLetters\makecal
\doLetters\makec
\doLetters\makescr
\doLetters\makebf
\doletters\makebf
\doletters\makegr

\begin{document}

\title[Functional equation for the $L$-function of a cubic order]{Beyond Endoscopy for $\mathrm{GL}_3(\mathbb{Q})$: Functional equation for the $L$-function of a cubic order}

\keywords{}

\subjclass[2020]{
11F66, 
11F72, 
11R54
}

\author[Yuchan Lee]{Yuchan Lee}
\thanks{This work was supported by the National Research Foundation of Korea (NRF) grant funded by the Korea government (MSIT) (No. RS-2026-25508638)}

\address{Yuchan Lee \\  Department of Mathematics, POSTECH, 77, Cheongam-ro, Nam-gu, Pohang-si, Gyeongsangbuk-do, 37673, KOREA}

\email{yuchanlee329@gmail.com}

\begin{abstract}
The Beyond Endoscopy strategy, proposed by Langlands, aims to establish the principle of functoriality by analyzing the trace formula. 
Recently, Deng and Espinosa advanced this program for $\mathrm{GL}_3(\mathbb{Q})$ by isolating the contribution of the trivial representation from the elliptic regular part. 
Their work relies on a conjectural factorization formula for the  $L$-function associated with a cubic order, which yields the  functional equation for the completed $L$-function. 
In this paper, we provide an unconditional proof of this functional equation for every Gorenstein order in a cubic number field. 
As a consequence, their isolation of the trivial representation for $\mathrm{GL}_3(\mathbb{Q})$ becomes fully unconditional.
\end{abstract}
\maketitle
\tableofcontents

\section{Introduction}
Beyond Endoscopy, proposed by Langlands \cite{Lang04}, is a strategy for establishing the general principle of functoriality. 
It consists of two primary steps. 
The first step, by means of the Arthur-Selberg trace formula, is to isolate packets of cuspidal automorphic representations according to the order of the pole of the associated $L$-functions.
The second step compares these isolated spectral contributions across different reductive groups to establish functorial transfers. 

This strategy was first carried out successfully in the series of papers \cite{Alt1}, \cite{Alt2}, and \cite{Alt3} of Altu\u{g} for $\mathrm{GL}_2(\mathbb{Q})$.
In particular, \cite{Alt1} concerns the first step for $\mathrm{GL}_2(\Q)$. 
Applying the \textit{approximate functional equation} and \textit{Poisson summation}, Altu\u{g} explicitly isolated the contribution of the trivial representation from the elliptic regular part of the trace formula.
Recently, Deng and Espinosa \cite{DE} generalized the results of \cite{Alt1} to $\mathrm{GL}_3(\mathbb{Q})$, conditional on a specific conjecture associated with the $L$-function of a cubic order (see Conjecture \ref{conj of DE}).
This conjecture is needed to yield the necessary functional equation, stated in Theorem \ref{intro:main thm}.

In the present paper, we observe that the full strength of this conjecture is not necessary for their deduction. 
Instead, we provide an unconditional proof of the functional equation for the $L$-function of a cubic order. 
Consequently, the isolation of the contribution of the trivial representation from the elliptic regular part of the trace formula for $\mathrm{GL}_3(\mathbb{Q})$ achieved in \cite{DE} now holds unconditionally.

\subsection{A brief review of Deng and Espinosa's work}
As explained in \cite[Section 1.1]{DE} (see also \cite[Section 1.2]{Alt1}), an essential prerequisite for the first step of Beyond Endoscopy is to analytically manipulate the elliptic regular part to separate and cancel the contribution of the trivial representation. 
More precisely, one seeks a concrete expansion for the following difference:
\begin{equation}\label{eq:LHS for beyond endoscopy}
\mathrm{I}_{\mathrm{ell}}(f)-\mathrm{Tr}(\mathbf{1}(f)),
\end{equation}
for a suitable test function $f$ as given in \cite[Section 2.0.1]{DE}.
Here, $\mathrm{I}_{\mathrm{ell}}(f)$ denotes the contribution from elliptic regular conjugacy classes on the geometric side of the Arthur-Selberg trace formula, 
\[
I_{\mathrm{ell}}(f)
=
\sum_{[\gamma]}\mathrm{vol}(\gamma)\mathcal O_\gamma(f),
\]
where \(\mathcal O_\gamma(f)\) is the orbital integral associated with the conjugacy class $[\gamma]$, and $\mathrm{vol}(\gamma)$ is the corresponding volume factor.
The term $\mathrm{Tr}(\mathbf{1}(f))$ is the contribution of the trivial representation on the spectral side.
We refer to \cite[Section 1.1]{DE} for the precise definitions.

Although a brief strategy for obtaining the concrete expansion of \eqref{eq:LHS for beyond endoscopy} is provided in \cite[Section 1.2]{DE}, we briefly outline the derivation here to clarify the precise role of our main result.
The chosen test function $f$, which depends on a fixed prime $p$ and a positive integer $k$, restricts the summation to conjugacy classes with characteristic polynomials of the form $X^3-aX^2+bX\pm p^k$.
Consequently, this yields the following expression for $\mathrm{I}_{\mathrm{ell}}(f)$:
\[
\mathrm{I}_{\mathrm{ell}}(f)=\sum_{\pm}\sum_{(a,b)\in V_3(\pm)} \mathrm{vol}(\gamma(a,b))\mathcal{O}_{\gamma(a,b)}(f)
\]
where $V_3(\pm)$ is the set of integral pairs $(a,b)$ such that $X^3-aX^2+bX\mp p^k$ is irreducible, and $\gamma(a,b)$ denotes the conjugacy class associated with this cubic polynomial.



For $\mathrm{GL}_2(\Q)$, Altu{\u{g}} reformulated the elliptic regular contribution $\mathrm{I}_{\mathrm{ell}}(f)$ in terms of the arithmetic function \[
L(s,\delta)=\sum_{d^2|\delta}'\frac{1}{d^{2s-1}} \sum_{l=1}^\infty\frac{1}{l^s}\left(\frac{\delta/d^2}{l}\right),
\]
where the $'$ in the sum indicates that it is restricted to $d$ such that $\delta/d^2\equiv 0\textit{ or }1\ (\mathrm{mod}\ 4)$.
Since its completed $L$-function satisfies a functional equation (see \cite[Proposition 3.1]{Alt1}), the \textit{approximate functional equation} becomes applicable.
This, together with the explicit Dirichlet series expansion of $L(s,\delta)$, allows the relevant summation to be extended to the full lattice, making Poisson summation applicable. The resulting identity yields an explicit expansion of $\mathrm{I}_{\mathrm{ell}}(f)-\mathrm{Tr}(\mathbf1(f))$.

To extend Altu\u{g}'s strategy, it is essential to obtain a suitable generalization of $L(s,\delta)$ together with its functional equation.
Arthur \cite{Art18} addressed this problem and proposed that the Dedekind zeta function $J_R(s)$ (see Conjecture \ref{conj of DE}) of an order $R$ introduced in \cite{Yun13} would provide the desired generalization.
This was verified for $\mathrm{GL}_2(\mathbb{Q})$ in \cite{Es23} and \cite{Es26}.
As a consequence, one obtains that $L(s,\mathrm{Tr}(\gamma)^2-4\det(\gamma))$ coincides with $J_{\Z[\gamma]}(s)/\zeta_\mathbb{Q}(s)$.
Then, the functional equation for the completion of $L(s,\mathrm{Tr}(\gamma)^2-4\det(\gamma))$ follows from the functional equation for the completed zeta function associated with $J_{\Z[\gamma]}(s)$ proved in \cite[Theorem 1.2.(1)]{Yun13}.

Motivated by this perspective, Deng and Espinosa introduced the following cubic analogue of $L(s,\delta)$:
\begin{definition}[{\cite[Definition 12]{DE}}]\label{intro:def:L_function_of_order}
Let $R$ be an order in a cubic number field $E$. We define
\begin{equation*}
    L(s, R)=\sum_{R\subset \mathcal{O}\subset\mathcal{O}_E} h_{\mathcal{O}}(s)\, L_\mathcal{O}(s)[\mathcal{O}: R]^{1-2s},
\end{equation*}
where the summation runs over all overorders $\mathcal{O}$ of $R$, and the terms $L_\mathcal{O}(s)$ and $h_\mathcal{O}(s)$ are defined in Definitions \ref{def:term L}-\ref{def:term h}, respectively.
We define its completion as
\begin{equation}\label{intro:definition of complete Lfunction for R}
    \Lambda(s, R)=\frac{\Gamma_{R,\infty}(s)}{\pi^{-s/2}\Gamma(s/2)}L(s,R),
\end{equation}
where $D_R$ is the absolute discriminant of $R$ and
\begin{equation*}
    \Gamma_{R,\infty}(s)=D_{R}^{s/2}(\pi^{-s/2}\Gamma(s/2))^{r_1}((2\pi)^{1-s}\Gamma(s))^{r_2}.
\end{equation*}
\end{definition}
This definition is designed to retain an explicit Dirichlet series while providing the expected higher-rank analogue of the relationship between $L(s,\delta)$ and $J_{\mathbb{Z}[\gamma]}(s)$.
This expectation is made precise in the following conjecture.
\begin{conjecture}[{\cite[Conjecture A]{DE}}]\label{conj of DE}
    Let $R$ be a Gorenstein order over $\Z$. Let $R^\vee=\mathrm{Hom}(R,\Z)$, and let\[
J_R(s)=\sum_{M\subset R^\vee} [R^\vee:M]^{-s}\]be Yun's Dedekind zeta function in \cite{Yun13}.    
    Then we have
    \[
    J_R(s)=L(s,R)\zeta_\Q(s).
    \]
\end{conjecture}

Assuming Conjecture \ref{conj of DE}, they proved the functional equation for the completed $L$-function $\Lambda(s,R)$ for each $R=\mathbb{Z}[\gamma(a,b)]$, making it possible to invoke the \textit{approximate functional equation}.
Using 
the explicit Dirichlet series expansion of $L(s,R)$, they 
extend the $(a,b)$-sum in $\mathrm{I}_{\mathrm{ell}}(f)$ to the full lattice $\mathbb{Z}^2$, making it possible to apply \textit{Poisson summation}.
The isolation of the contribution of the trivial representation then follows from an explicit evaluation of the resulting Kloosterman sums.

\subsection{Main Theorem: the functional equation for $L(s,R)$}\label{sec:intro:main thm}
It is worth emphasizing that the primary role of Conjecture \ref{conj of DE} in \cite{DE} is to obtain the functional equation for $L(s,R)$. 
Indeed, the conjecture is used only in \cite[Propositions 19-20]{DE}, where Proposition 20 depends solely on the fact that $L(1,R)$  coincides with $\left.\frac{J_R(s)}{\zeta_\Q(s)}\right|_{s=1}$ - this is verified in \cite[Proposition 18]{DE} (note that this is proven without assuming Conjecture \ref{conj of DE}).
Moreover, the local factor of $J_R(s)$ is expressed as a weighted sum of local integrals over $(R\otimes_\mathbb{Z}\mathbb{Q}_p)^\times$ (see \cite[Lemma 2.10]{Yun13}), making a direct verification of the conjecture technically involved. 
Our approach avoids this difficulty entirely.
Instead, following \cite{CHL2}, we exploit the ideal class monoid structure of $R$ to enumerate the local factor of $L(s,R)$ at each $p$ explicitly as a polynomial in $p^{-s}$.

The main theorem of this paper is the following unconditional functional equation for $\Lambda(s,R)$, corresponding to \cite[Proposition 19]{DE}:
\begin{theorem}[{Theorem \ref{thm:main thm}}]\label{intro:main thm}
    For a Gorenstein order $R$ of a cubic number field $E$, we have
    \[
    \Lambda(s,R)=\Lambda(1-s,R).
    \]
\end{theorem}\noindent
Therefore, our result makes the main result of \cite{DE}, namely the isolation of $\mathrm{Tr}(\mathbf{1}(f))$ from $\mathrm{I}_{\mathrm{ell}}(f)$, unconditional by removing the assumption that Conjecture \ref{conj of DE} holds.

\subsection{The strategy of the proof}
We first reduce the proof of the functional equation for $\Lambda(s,R)$ to that for a more manageable function $\tilde{L}(s,R)$, introduced in Lemma \ref{lem:definition of tilde L}.
By Proposition \ref{prop:euler_product}, $\tilde{L}(s,R)$ admits the Euler product
\[
\tilde{L}(s,R)=\prod_p \tilde{L}_p(s,R).
\]
Moreover, by Remark \ref{rmk:finite alive for local factors}, we have $\tilde{L}_p(s,R)=1$ for all but finitely many primes $p$.
Therefore, it suffices to prove the local functional equation for each local factor $\tilde{L}_p(s,R)$, which depends only on the localized order $R_p=R\otimes_\mathbb{Z}\mathbb{Z}_p$.

In the case that $R_p$ is a quadratic $\mathbb{Z}_p$-order, then $R_p$ is a Bass order, i.e., every overorder of $R_p$ is Gorenstein, by \cite{CHL}.
Moreover, Example \ref{eg:when n=2} shows that the set of overorders of $R_p$ is totally ordered and 
\begin{equation}\label{eq:L function in 2}
\tilde{L}_p(s,R)= p^{S(R_p)s}\Bigg(p^{S(R_p)(1-2s)}+\sum_{i=1}^{S(R_p)}\left(1-\frac{\chi_E(p)}{p^s}\right)p^{(1-2s)(S(R_p)-i)}\Bigg)
\end{equation}
where $\chi_E(p)=\left\{\begin{array}{l l}
    -1 & p\textit{ is inert in $E$};  \\
    0& p\textit{ ramifies in $E$}; \\
    1  & p\textit{ splits in $E$}.
\end{array}\right.$
On the other hand, when $R_p$ is a cubic order, the set of overorders is no longer totally ordered (see \cite[Examples 4.8 and 4.14]{CHL2}) and not every overorder of $R_p$ is Gorenstein.
Consequently, this makes the enumeration of $\tilde{L}_p(s,R)$ complicated.

To overcome this difficulty, we propose candidate expressions motivated from the right-hand side of \eqref{eq:L function in 2}.
Theorems \ref{thm:reformulation_for_unram}, \ref{thm:reformulation_for_ram}, \ref{thm:reformulation_for_unramlin}, \ref{thm:reformulation_for_ramlin}, and \ref{thm:reformulation_for_split} formulate these candidates and prove that they coincide with the corresponding local factors, according to the splitting type of $E$ at $p$ (see \eqref{eq:splittingtype}).
The proof proceeds by rewriting both $\tilde{L}_p(s,R)$ and the corresponding candidate expression as sums of monomials of the form $p^{A(1-2s)+B(2-3s)}$,
and then showing that the associated multisets of exponent pairs $(A,B)$ coincide exactly.

With these explicit formulas in hand, we obtain the local functional equations depending on the splitting type of $E$ at $p$. For primes where $E$ is irreducible, the functional equation is proved by applying a combinatorial identity for finite double sums (see Lemma \ref{lem:double sum formula for nonsplit case}). 
For primes where $E$ splits, we first construct a central algebraic identity for the base splitting type $(1\ 1^2)$ (see \eqref{eq:Key_equation}), and then deduce the functional equations for the remaining split cases simultaneously through  algebraic reductions.



\vspace{0.5em}\noindent
\textbf{Acknowledgment.} 
The author would like to express gratitude to Sug Woo Shin for introducing the concept of Beyond Endoscopy and inspiring interest in related topics. 
The author would also like to thank Sungmun Cho and Jungtaek Hong for various discussions.

\section{Notation and Preliminaries}

Let $Z=\mathbb{Z}$ or $\mathbb{Z}_p$, and let $Q$ be its field of fractions, so that $Q=\mathbb{Q}$ or $\mathbb{Q}_p$, respectively.
Let $K$ be an \'etale $Q$-algebra.
\begin{itemize}
\item An order of $K$ is a subring $\Mfo$ of $K$ such that $\Mfo$ is a finitely generated $Z$-module containing $Z$ and such that $\Mfo\otimes_ZQ\cong K$. 
     An order $\Mfo'$ of $K$ is called an overorder of $\Mfo$ if $\Mfo \subset \Mfo'$. 

\item The maximal order of $K$ is denoted by $\mathcal{O}_K$. 
The existence and uniqueness of $\Mfo_K$ are explained in \cite[the first paragraph of Section 2]{Ma20} or \cite[the third paragraph of Section 2.2]{Ma24}.
Note that if $K$ is a field, then its maximal order $\mathcal{O}_K$ is the ring of integers of $K$.
More generally, since $K$ is an \'etale $\mathbb{Q}$-algebra, it decomposes as a finite product of fields
$K=\prod_i K_i$, and accordingly its maximal order decomposes as
$\mathcal{O}_K=\prod_i\mathcal{O}_{K_i}$.
     
\item    
 A fractional $\Mfo$-ideal $I$ is a finitely generated $\Mfo$-submodule of $K$ such that $I\otimes_ZQ\cong K$. 
 \item 
 The ideal class monoid of $\mathcal{O}$ is defined to be the monoid of equivalence classes of fractional $\mathcal{O}$-ideals up to multiplication by an element of $K^\times$ (see \cite{CHL} or \cite{CHL2}).
Note that fractional ideals of $\mathcal{O}$ need not be invertible unless $\mathcal{O}$ is maximal, which naturally yields a monoid structure rather than a group.

\item The conductor $\mathfrak{f}(\Mfo)$ of an order $\Mfo$ is the largest ideal of $\Mfo_K$ which is contained in $\Mfo$. In other words, $\mathfrak{f}(\Mfo)=\{a\in \Mfo_K\mid a\Mfo_K\subset \Mfo\}$.

\item For $e_1,\cdots, e_n\in \mathcal{O}_K$, we denote by $Z\langle e_1,\cdots ,e_n\rangle$  the $Z$-linear span of $\{e_1,\cdots, e_n\}$ in $\mathcal{O}_K$.

\item For orders $\mathcal{O}_1$ and $\mathcal{O}_2$ in $K$ such that $\mathcal{O}_2\subset \mathcal{O}_1$,
we denote their index by $[\mathcal{O}_1:\mathcal{O}_2]=\#(\mathcal{O}_1/\mathcal{O}_2)$.
\end{itemize}

\vspace{0.5em}\noindent
In particular, for a number field $E$ and an order $R$ of $E$, we define the following notions.
\begin{itemize}
    \item 
Let $R_p\cong R\otimes_\Z\Z_p$ be the $p$-adic completion of $R$, and let $E_p\cong E\otimes_{\Q}\Q_p$ (or equivalently, $E_p\cong R_p\otimes_{\mathbb{Z}_p}\Q_p$) be the ring of fractions of $R_p$.
\item For an overorder $\mathcal{O}$ of $R$, note that $[\mathcal{O}:R]=\prod_p[\mathcal{O}_p:R_p]$ (so that $[\mathcal{O}_p:R_p]=1$ for all but finitely many primes $p$).
We denote by $S(R_p)$ the exponent of $p$ in $[\mathcal{O}_{E_p}:R_p]$, so that $[\mathcal{O}_E:R]=\prod_p p^{S(R_p)}$. 
We refer to this relation as the local-global index formula.
\item
The trace pairing $(x,y)\mapsto \mathrm{Tr}_{E/\Q}(xy)$ allows us to identify $E$ and the $\Q$-linear dual of $E$. 
In particular, the dual $R^\vee=\mathrm{Hom}_\Z(R,\Z)$ can also be viewed as a fractional ideal in $E$ which contains $R$.
The absolute discriminant $D_R$ of $R$ is defined to be $[R^\vee:R]$ (see \cite[Section 3.2]{Yun13}).
\end{itemize}

\section{$L$-function for an order}
In this section, we introduce the notion of the $L$-function for an  order.
While this framework applies generally, we focus primarily on the case of cubic orders, adapting the formulation established in \cite{DE}.
\begin{definition}\label{def:term L}
Let $p$ be a prime, and let $\mathcal{O}$ be an order in an \'{e}tale algebra over $\Q$. 
Let $\{\mathfrak{m}_i\}_{i=1}^r$ be the maximal ideals of $\mathcal{O}_{p}$, and let $k_i=\mathcal{O}_p/\mathfrak{m}_i$ be the corresponding residue fields. 
\begin{enumerate}
    \item (\cite[Definition 8]{DE}) We define
\[\zeta_{\mathcal{O}_p}(s)\;=\;\prod_{i=1}^r \frac{1}{1-|k_i|^{-s}},\text{ and }
\zeta_{\mathcal{O}}(s)=\prod_{p}\zeta_{\mathcal{O}_p}(s).\]
\item(\cite[Definition 9]{DE}) We define
\[
L_{\mathcal{O}_p}(s) =\frac{\zeta_{\mathcal{O}_p}(s)}{\zeta_{\Q_p}(s)},\text{ and }
L_{\mathcal{O}}(s)=\prod_{p}L_{\mathcal{O}_p}(s).
\]
\end{enumerate}
\end{definition}

\begin{definition}[{\cite[Definition 10]{DE}}]\label{def:term h}
Let $p$ be a prime. Given an order $\mathcal{O}$ in an \'{e}tale algebra over $\Q$, we define 
\[
h_{\mathcal{O}_p}(s)=\begin{cases}
1, &\text{\, if $\mathcal{O}_p$ is Gorenstein. }\\
1+p^{1-s},  &\text{\, otherwise. }\\
\end{cases}
\]  
Globally, we define the multiplicative function 
\begin{equation*}
    h_{\mathcal{O}}(s)=\prod_{p}h_{\mathcal{O}_p}(s).
\end{equation*}
\end{definition}
Based on the above two definitions, we define the $L$-function of an order $R$ in a number field $E$.
Here, we recall that the maximal order $\mathcal{O}_E$ in $E$ coincides with the ring of integers of $E$. 
\begin{definition}[{\cite[Definition 12]{DE}}]\label{def:L_function_of_order}
Let $R$ be a $\Z$-order inside the maximal order $\mathcal{O}_E$. We define
\begin{equation*}
    L(s, R)=\sum_{R\subset \mathcal{O}\subset\mathcal{O}_E} h_{\mathcal{O}}(s)\, L_\mathcal{O}(s)[\mathcal{O}: R]^{1-2s},
\end{equation*}
where the summation runs over all overorders $\mathcal{O}$ of $R$.
We define its completion as
\begin{equation}\label{eq:definition of complete Lfunction for R}
    \Lambda(s, R)=\frac{\Gamma_{R,\infty}(s)}{\pi^{-s/2}\Gamma(s/2)}L(s,R),
\end{equation}
where $D_R$ is the absolute discriminant of $R$ and
\begin{equation}\label{eq:Gamma factor for R}
    \Gamma_{R,\infty}(s)=D_{R}^{s/2}(\pi^{-s/2}\Gamma(s/2))^{r_1}((2\pi)^{1-s}\Gamma(s))^{r_2}.
\end{equation}
\end{definition}

The main goal of this paper is to unconditionally obtain the functional equation for $\Lambda(s,R)$ where $R$ is a cubic order:
\[
\Lambda(s,R)=\Lambda(1-s,R).
\]
In practice, we will work with a more manageable function. The following lemma justifies this reduction by allowing us to study this modified function instead of $\Lambda(s,R)$.

\begin{lemma}\label{lem:definition of tilde L}
    To prove the functional equation for $\Lambda(s,R)$, it suffices to show that
    \[
    \tilde{L}(s,R)=\tilde{L}(1-s,R)
    \]where $\tilde{L}(s,R)= [\mathcal{O}_E:R]^sL(s,R)\frac{\zeta_\mathbb{Q}(s)
    }{\zeta_E(s)}$. 
\end{lemma}
\begin{proof}  
Note that the completed Dedekind zeta function $\Lambda_E(s)$ of the number field $E$ is given by
\[
\Lambda_E(s)=|\Delta_E|^{s/2}(\pi^{-s/2}\Gamma(s/2))^{r_1}((2\pi)^{1-s}\Gamma(s))^{r_2}\zeta_E(s)=\frac{\Gamma_{R,\infty}(s)\zeta_E(s)}{[\mathcal{O}_E:R]^s}
\]where $\Delta_E$ denotes the discriminant of $E$.
Here, the second equality follows from the relation $D_{R}=[\mathcal{O}_E: R]^2 |\Delta_E|$ by \cite[Section 3.2]{Yun13}.
We then have
    \[\tilde{L}(s,R)=\frac{[\mathcal{O}_E:R]^s}{\zeta_E(s)}L(s,R)\zeta_\mathbb{Q}(s)=
    \frac{\Gamma_{R,\infty}(s)L(s,R)\zeta_\mathbb{Q}(s)}{ \Lambda_E(s)}.
    \]
    Suppose that the equation $\tilde{L}(s,R)=\tilde{L}(1-s,R)$ holds. 
    Since the completed Dedekind zeta function satisfies $\Lambda_E(s)=\Lambda_E(1-s)$, this directly yields
    \[
     \Gamma_{R,\infty}(s)L(s,R)\zeta_\mathbb{Q}(s)=\Gamma_{R,\infty}(1-s)L(1-s,R)\zeta_{\mathbb{Q}}(1-s).
    \]
    Dividing both sides by the completed Riemann zeta function $\Lambda_\mathbb{Q}(s)=\pi^{-s/2}\Gamma(s/2)\zeta_\Q(s)$, which satisfies $\Lambda_\Q(s)=\Lambda_\Q(1-s)$, we immediately recover the desired functional equation for $\Lambda(s,R)$.
\end{proof}
\begin{proposition}\label{prop:euler_product}
We have the following Euler product expansion:
    \[
    \tilde{L}(s,R)=\prod_{p}p^{S(R_p)s}\sum_{R_p\subset \mathcal{O}_p\subset \mathcal{O}_{E_p}} h_{\mathcal{O}_p}(s)\frac{\zeta_{\mathcal{O}_p}(s)}{\zeta_{E_p}(s)} [\mathcal{O}_p:R_p] ^{1-2s}.
    \]
\end{proposition}
\begin{proof}
    By substituting the definitions of $h_\mathcal{O}(s)$ and $L_\mathcal{O}(s)$ for an overorder $\mathcal{O}$ of $R$, we have
    \[
    L(s,R)=\sum_{R\subset \mathcal{O}\subset \mathcal{O}_E}\prod_{p}h_{\mathcal{O}_p}(s)\frac{\zeta_{\mathcal{O}_{p}}(s)}{\zeta_{\mathbb{Q}_p}(s)}[\mathcal{O}_p:R_p]^{1-2s}.
    \]
    Recalling that $\tilde{L}(s,R)=[\mathcal{O}_E:R]^s L(s,R)\prod_{p}\frac{\zeta_{\mathbb{Q}_p}(s)}{\zeta_{E_p}(s)}$, this yields
    \[
    \tilde{L}(s,R)=[\mathcal{O}_E:R]^s\sum_{R\subset \mathcal{O}\subset \mathcal{O}_E}\prod_{p}h_{\mathcal{O}_p}(s)\frac{\zeta_{\mathcal{O}_p}(s)}{\zeta_{E_p}(s)}[\mathcal{O}_p:R_p]^{1-2s}.
    \]
    For any overorder $\mathcal{O}$ of $R$, we know that $[\mathcal{O}_p:R_p]=1$ for all but finitely many primes $p$.
    Moreover, the local-global index formula yields  $[\mathcal{O}_E:R]=\prod_p p^{S(R_p)}$.
    Therefore, to express the summation over global overorders as an Euler product over $p$, it suffices to show that the following map is bijective:
    \begin{equation}\label{eq:bijection in R}
    \{\textit{overorders of }R\} \rightarrow \prod_{p}\{\textit{overorders of }R_p\},\ \mathcal{O}\mapsto (\mathcal{O}_p)_{p}.
    \end{equation}
    By \cite[Remark B.(1)]{CHL}, we have the algebra decomposition $R_p\cong \bigoplus_{v|p, v\in |R|}R_v$ where $|R|$ denotes the set of maximal ideals of $R$ and $R_v$ is the $v$-adic completion of $R$.
    Since the overorders of a finite direct sum of rings naturally correspond to the direct product of the overorders of its components, this induces a canonical bijection 
    \[
    \prod_p\{\textit{overorders of }R_p\}\cong \prod_p\prod_{v|p, v\in|R|}\{\textit{overorders of }R_v\}=\prod_{v\in |R|}\{\textit{overorders of $R_v$}\},  
    \]where $(\mathcal{O}_p)_p$ maps to  $((\mathcal{O}_p\otimes_{R_p}R_v)_{v|p, v\in|R|})_p$.  
    Finally, \cite[Corollary 5.5.(1)]{CHL} states that the map $\mathcal{O}\mapsto (\mathcal{O}\otimes_R R_v)_{v\in |R|}$ gives a bijection between $\{\textit{overorders of $R$}\}$ and $\prod_{v\in|R|}\{\textit{overorders of $R_v$}\}$.
    Since we have \[\mathcal{O}\otimes_R R_v\cong \mathcal{O}\otimes_R(R_p\otimes_{R_p}R_v)\cong  (\mathcal{O}\otimes_R R_p)\otimes_{R_p}R_v
    \cong \mathcal{O}_p\otimes_{R_p}R_v\] for each $v\in|R|$,
    the map \eqref{eq:bijection in R} is indeed bijective which completes the proof.
\end{proof}

For ease of exposition and subsequent references, we define the local factor of $\tilde{L}_p(s,R)$ as follows.
\begin{definition}\label{def:local factor of L function}
    We define the local factor $\tilde{L}_p(s,R)$ of $\tilde{L}(s,R)$ at a prime $p$ as follows
    \[
    \tilde{L}_p(s,R)=p^{S(R_p)s}\sum_{R_p\subset \mathcal{O}_p\subset \mathcal{O}_{E_p}} h_{\mathcal{O}_p}(s)\frac{\zeta_{\mathcal{O}_p}(s)}{\zeta_{E_p}(s)}[\mathcal{O}_p:R_p]^{1-2s},
    \]so that $\tilde{L}(s,R)=\prod_p \tilde{L}_p(s,R)$.
\end{definition}
\begin{remark}\label{rmk:finite alive for local factors}
    For all but finitely many primes $p$, we have $S(R_p) = 0$, or equivalently, $R_p = \mathcal{O}_{E_p}$. 
    This implies that $\tilde{L}_p(s,R) = 1$. 
Therefore, to verify the functional equation for the global function $\tilde{L}(s,R)$, it suffices to show that the local functional equation $\tilde{L}_p(s,R)=\tilde{L}_p(1-s,R)$ holds for each prime $p$.
\end{remark}
\begin{example}\label{eg:when n=2}
    Suppose that $E$ is a quadratic field.
    Then, any order $R$ in $E$ is a Bass order (i.e., every overorder of $R$ is Gorenstein) by \cite{CHL}.
    Moreover, the set of overorders of $R_p$ is totally ordered:
    \begin{itemize}
        \item
        Suppose that $p$ does not split in $E$. 
        Then $R_p=\mathbb{Z}_p[p^{S(R_p)}x]$ for some element $x\in\mathcal{O}_{E_p}$ such that $\mathcal{O}_{E_p}=\mathbb{Z}_p[x]$. 
        In particular, if $p$ ramifies in $E$, then the element $x$ must be chosen to be a uniformizer of $E_p$. 
        By \cite[Theorem 3.11 and Proposition 3.16]{CHL}, the overorders of $R_p$ are given by the chain:
\[
R_p=\mathbb{Z}_p[p^{S(R_p)}x]
\subset
\mathbb{Z}_p[p^{S(R_p)-1}x]
\subset
\cdots
\subset
\mathbb{Z}_p[x]
=
\mathcal{O}_{E_p}.
\]
        \item Suppose that $p$ splits in $E$. We identify $\mathcal{O}_{E_p}$ with $\mathbb{Z}_p\times \mathbb{Z}_p$ via the isomorphism $E_p\cong \mathbb{Q}_p\times \mathbb{Q}_p$. 
        Under this identification, $R_p$ corresponds to $\mathbb{Z}_p\langle (1,1),(p^{S(R_p)},0) \rangle$.
        By \cite[Theorem 6.11]{CHL}, the overorders of $R_p$ form the chain:
        \[
        R_p=\mathbb{Z}_p\langle (1,1),(p^{S(R_p)},0) \rangle\subset \mathbb{Z}_p\langle (1,1),(p^{S(R_p)-1},0) \rangle\subset  \cdots \subset \mathbb{Z}_p\langle (1,1),(1,0) \rangle=\mathcal{O}_{E_p}.\]
    \end{itemize}
    In either case, $\tilde{L}_p(s,R)$ can be explicitly evaluated as the following polynomial in $p^{-s}$:
    \begin{equation}\label{eq:n=2 L-function in eg}
        \tilde{L}_p(s,R)= p^{S(R_p)s}\Bigg(p^{S(R_p)(1-2s)}+\sum_{i=1}^{S(R_p)}\left(1-\frac{\chi_E(p)}{p^s}\right)p^{(1-2s)(S(R_p)-i)}\Bigg)
\end{equation}
where $\chi_E(p)=\left\{\begin{array}{l l}
    -1 & p\textit{ is inert in $E$};  \\
    0& p\textit{ ramifies in $E$};\\
     1  & p\textit{ splits in $E$}.
\end{array}\right.$
\end{example}
\section{Reformulation of local factor $\tilde{L}_p(s,R)$}\label{sec:reformulation of local factors}
For the remainder of this paper, we fix a cubic number field $E$ and a Gorenstein order $R$ in $E$. 
As we mentioned in the introduction, unlike the case of a quadratic order in Example \ref{eg:when n=2}, the set of overorders is not totally ordered and not every overorder is Gorenstein.
In this section, we reformulate the local function $\tilde{L}_p(s,R)$ of Definition \ref{def:local factor of L function} into an explicit polynomial in $p^{-s}$, which naturally extends the right-hand side of \eqref{eq:n=2 L-function in eg}.
Our enumeration is based on the ideal class monoid structure of $R$ previously investigated by the author in \cite{CHL2}. 

We denote the splitting type $\sigma_p(E)$ of $E$ at prime $p$ as follows:
\begin{equation}\label{eq:splittingtype}
\sigma_p(E)=\left\{
\begin{array}{l l}
(3) &\text{if $E_p/ \mathbb{Q}_p$ is an unramified field extension,  }\\
(1^3) &\text{if $E_p/\mathbb{Q}_p$ is a ramified field extension,}\\
(1\ 2) &\text{if $E_p\cong \mathbb{Q}_p\times E'_p$ with $E'_p/\mathbb{Q}_p$ an unramified quadratic field extension,  }\\
(1\ 1^2) &\text{if $E_p\cong \mathbb{Q}_p\times E'_p$ with $E'_p/\mathbb{Q}_p$ a ramified quadratic field extension, }\\  
(1\ 1\ 1) &\text{if $E_p\cong \mathbb{Q}_p\times \mathbb{Q}_p \times \mathbb{Q}_p$.}
\end{array}
\right.
\end{equation}
If $\sigma_p(E)=(3)$ or $(1^3)$, we say that $E$ is \textit{irreducible} at p. 
Otherwise, we say that $E$ \textit{splits} at p.

Since $R$ is Gorenstein, the corresponding local order $R_p$ is also Gorenstein for each prime $p$ by the second product formula in \cite[Proposition 2.6]{CHL2}.
Furthermore, the following lemma allows us to regard $R_p$ as a simple extension over $\mathbb{Z}_p$ whenever $\sigma_p(E)\neq (1\ 1\ 1)$.
\begin{lemma}
    Suppose that $\sigma_p(E)\neq (1\ 1\ 1)$. $R_p$ is Gorenstein if and only if $R_p$ is a simple extension over $\mathbb{Z}_p$.
\end{lemma}
\begin{proof}
    \cite[Proposition 3.1]{CHL2} directly yields this lemma.
\end{proof}

\subsection{The case where $E$ is irreducible at $p$}
\subsubsection{The case where $\sigma_p(E)=(3)$}
Since $R_p$ is a simple extension over $\mathbb{Z}_p$ and $E_p/\mathbb{Q}_p$ is an unramified field extension, $R_p$ takes the form $\mathbb{Z}_p[p^d x]$ for some integer $d\geq 0$ and some element $x\in E_p$ such that $\mathcal{O}_{E_p}=\mathbb{Z}_p[x]$. 
As $\mathbb{Z}_p$-modules, $\mathcal{O}_{E_p}$ is generated by $\{1,x,x^2\}$ whereas $R_p$ is generated by $\{1,p^d x,p^{2d}x^2\}$. Consequently, we have $S(R_p)=3d$, which implies that $3|S(R_p)$. 

\begin{proposition}\label{prop:first formula for unram}
Suppose that $\sigma_p(E)=(3)$ and let $d=S(R_p)/3$. Then we have 
 \begin{align*}
    \tilde{L}_p(s,R)=p^{3ds}\Big(
    p^{3d(1-2s)}+
    (1+p^{-s}+p^{-2s})P_{(3)}^{3d}(s)
    \Big)
    \end{align*}where
    \begin{align*}
        P_{(3)}^{3d}(s)=&\sum_{1\leq f\leq 2d,\ 2|f}p^{\min(d,f)-\frac{1}{2}f+(3d-\frac{3}{2}f)(1-2s)}+\sum_{1\leq f\leq d,\ 2|f}p^{\frac{1}{2}f-1}p^{(3d-\frac{3}{2}f)(1-2s)}\\
    &+\sum_{1\leq f\leq 2d,}\sum_{\frac{3}{2}f<\alpha \leq f+\min(d,f)}(1+p^{1-s})p^{\min(d,f)-(\alpha-f)+(3d-\alpha)(1-2s)}\\
    &+\sum_{1\leq f\leq d}\sum_{\frac{3}{2}f<\alpha<2f}(1+p^{1-s})p^{2f-\alpha-1}p^{(3d-\alpha)(1-2s)}.
    \end{align*}
\end{proposition}
\begin{proof}
By the above argument, $R_p$ takes the form $\mathbb{Z}_p[p^d x]$ for some $x\in E_p$ such that $\mathcal{O}_{E_p}=\mathbb{Z}_p[x]$.
Then, \cite[Proposition 4.2 and Corollary 4.3]{CHL2} show that the set of overorders of $R_p=\mathbb{Z}_p[p^d x]$ is given by
\begin{equation}\label{eq:overorders of unram}
\{\Mfo_{\alpha,f,c}^1\}_{\substack{0\leq f \leq 2d,\\ \frac{3}{2}f \leq \alpha \leq f+\min(d,f),\\ c\in \mathbb{Z}_p/(p^{\min(d,f)-(\alpha-f)})}} \sqcup \{\Mfo_{\alpha,f,c}^2\}_{\substack{1\leq f\leq d,\\ \frac{3}{2}f\leq \alpha < 2f,\\ c\in \mathbb{Z}_p/(p^{2f-\alpha-1})}}
\end{equation}
where 
$\Mfo_{\alpha,f,c}^1=\mathbb{Z}_p\langle 1, p^{\alpha-f}(x+c p^{f-\min(d,f)}x^2),p^fx^2 \rangle$ 
and $\Mfo_{\alpha,f,c}^2=\mathbb{Z}_p\langle 1,p^fx,p^{\alpha-f}(cpx+x^2)\rangle$.
Moreover, by \cite[Corollary 4.3.(2) and (4)]{CHL2}, we have
 $S(R_p)=3d$, $[\mathcal{O}_{\alpha,f,c}^i:R_p]=p^{3d-\alpha}$, and $\mathcal{O}_{\alpha,f,c}^i$ is Gorenstein if and only if $\alpha=\frac{3}{2}f$ for $i=1,2$.
Recall the definition of $L_{\mathcal{O}_{\alpha,f,c}^i}(s)$ in Definition \ref{def:term L}.
Since the residue field of $\mathcal{O}_{\alpha,f,c}^i$ is $\mathbb{F}_p$, except the case where $i=1$ and $f=\alpha=0$ (so that $\mathcal{O}_{\alpha,f,c}^i=\mathcal{O}_{E_p}$), we have 
$\frac{\zeta_{\mathcal{O}_{\alpha,f,c}^i}(s)}{\zeta_{E_p}(s)}=\begin{cases}
    1 &\textit{if $i=1, f=\alpha=0$};\\
    1+p^{-s}+p^{-2s} &\textit{otherwise}.
\end{cases}$ 
Substituting these values into the definition of $\tilde{L}_p(s,R)$ (Definition \ref{def:local factor of L function}) and summing over the listed overorders immediately yields the desired formula.
\end{proof}

\begin{theorem}\label{thm:reformulation_for_unram}
Suppose that $\sigma_p(E)=(3)$ and let $d=S(R_p)/3$.
Then we have
\begin{align*}
        &\tilde{L}_p(s,R)\\&=
        p^{3ds}\Bigg(
        p^{3d(1-2s)}+(1+p^{-s}+p^{-2s})\sum_{d'=1}^{d}p^{3(d-d')(1-2s)}\Big(\sum_{i=0}^{d'}p^{(2-3s)i}+p^{1-2s}\sum_{i=0}^{d'-1}p^{(2-3s)i}+p^{2-4s}\sum_{i=0}^{d'-2}p^{(2-3s)i}\Big)\Bigg).
        \end{align*}
\end{theorem}
\begin{proof}
By Proposition \ref{prop:first formula for unram}, it suffices to show that
\begin{equation}\label{eq:comparision_unram}
P_{(3)}^{3d}(s)=\sum_{d'=1}^{d}p^{3(d-d')(1-2s)}\Big(\sum_{i=0}^{d'}p^{(2-3s)i}+p^{1-2s}\sum_{i=0}^{d'-1}p^{(2-3s)i}+p^{2-4s}\sum_{i=0}^{d'-2}p^{(2-3s)i}\Big).
\end{equation}
We first rewrite each term of $P_{(3)}^{3d}(s)$ in the form of $p^{A(1-2s)+B(2-3s)}$.
Then, by reorganizing the indices of summations, we have
\begin{equation}\label{unram form of AB}
    \begin{aligned}
    P_{(3)}^{3d}(s)=&\sum_{1\leq f\leq d,\ 2|f}p^{
        (3d-3f)(1-2s)+f(2-3s)}+\sum_{0\leq f< d,\ 2|f}p^{f(2-3s)}
        +\sum_{1\leq f\leq d,\ 2|f}p^{(3d-3f+3)(1-2s)+(f-2)(2-3s)}\\
    &+\sum_{0\leq  \alpha< \frac{d}{2}}\sum_{2\alpha< f \leq d}\Big(p^{
    (3d-2f-2\alpha)(1-2s)+2\alpha(2-3s)}+p^{(3d-2f-2\alpha-1)(1-2s)+(2\alpha+1)(2-3s)}\Big)\\
     &+\sum_{0\leq \alpha <\frac{d}{2}}\sum_{2d<f <3d-2\alpha}\Big(p^{(3d-f-2\alpha)(1-2s)+2\alpha(2-3s)}+p^{(3d-f-2\alpha-1)(1-2s)+(2\alpha+1)(2-3s)}\Big)\\
    &+\sum_{0< \alpha <\frac{d}{2}}\sum_{2\alpha<f\leq d}\Big(p^{
    (3d-2f-2\alpha+3)(1-2s)+(2\alpha-2)(2-3s)}
    +p^{(3d-2f-2\alpha+2)(1-2s)+(2\alpha-1)(2-3s)}\Big).
    \end{aligned}
    \end{equation}
    We represent the term $p^{A(1-2s)+B(2-3s)}$ by the pair of numbers $(3d-A,B)$.
    Since the representation of a monomial $p^{X+Ys}$ in the form $p^{A(1-2s)+B(2-3s)}$ is unique, it suffices to show that the corresponding multisets of pairs coincide.
    
    The three summations in the first line of \eqref{unram form of AB} are represented, for fixed $0\leq i \leq d$ such that $2|i$, by 
    \[
    \left\{
    \begin{array}{l l}
       \{(3,0), (3d,0)\} &\text{if $i=0$;} \\
       \{(3i,i), (3i+3,i), (3d,i)\} &\text{if $0< i<d-1$;}\\
       \{(3d-3,d-1), (3d,d-1)\} &\text{if $i=d-1$;}\\
       \{(3d,d)\} &\text{if $i=d$}.
    \end{array}
    \right.
    \]
    The remaining three double summations in \eqref{unram form of AB} are represented, 
    for fixed $i$ such that $0\leq i\leq  d$, by
    \[
    \left\{
    \begin{array}{l l}
    \{(3i+2,i),(3i+3,i),\cdots, (3d-1,i)\}\setminus \{(3i+3,i)\} &\text{if $i$ is even};\\
    \{(3i,i),(3i+1,i),\cdots,(3d,i)\}\setminus \{(3i+1,i)\}                             &\text{if $i$ is odd}.
    \end{array}\right.\]
    Taking the union of these contributions (for even $i$, the boundary terms supply $(3i,i)$, $(3i+3,i)$, and $(3d,i)$, which fill in the gaps of the double-summation contribution), $P_{(3)}(s)$ is the summation whose terms correspond to the set of pairs
    \[\{(x,0):2\leq x \leq 3d\}\cup \bigcup_{1\leq i\leq d}\Big(\{(3i,i)\}\cup \{(x,i):  3i+2 \leq x \leq 3d\}\Big).\]
    On the other hand, the right-hand side of (\ref{eq:comparision_unram}) corresponds to the set of pairs
    \begin{align*}
    &\bigcup_{d'=1}^d\Big( \{(3d',i)\}_{0\leq i\leq d'}\cup \{(3d'-1,i)\}_{0\leq i\leq d'-1} \cup \{(3d'-2,i)\}_{0\leq i\leq d'-2 } \Big) \\
        &=\bigcup_{i=0}^d \Big(\{(3d',i)\}_{\max(1,i)\leq d'\leq d}\cup \{(3d'-1,i)\}_{i+1\leq d'\leq d} \cup\{(3d'-2,i)\}_{i+2\leq d'\leq d}\Big)\\
    &=\{(x,0):2\leq x \leq 3d\}\cup \bigcup_{1\leq i\leq d}\Big(\{(3i,i)\}\cup \{(x,i):  3i+2 \leq x \leq 3d\}\Big).
    \end{align*}
   Hence the two multisets coincide, proving \eqref{eq:comparision_unram}.
\end{proof}
\subsubsection{The case where $\sigma_p(E)=(1^3)$}\label{sec:reformulation for ram}
Since $R_p$ is a simple extension over $\mathbb{Z}_p$ and $E_p$ is a ramified field extension, $R_p$ takes the form $\mathbb{Z}_p[p^d x]$ or $\mathbb{Z}_p[p^d x^2]$ for some integer $d\geq 0$ and a suitably chosen uniformizer $x$ of $E_p$.
If $R_p=\mathbb{Z}_p[p^d x]$, then we have $S(R_p)=3d$ by the same computation in the case $\sigma_p(E)=(3)$.
On the other hand, if $R_p=\mathbb{Z}_p[p^dx^2]$, then its basis as a $\mathbb{Z}_p$-module is given by $\{1,p^d x^2,p^{2d+1}x\}$ yielding $S(R_p)=3d+1$.
Consequently, we treat these two cases separately depending on whether $3|S(R_p)$ or $3|(S(R_p)-1)$.

\begin{proposition}\label{prop:first formula for ram}
Suppose that $\sigma_p(E)=(1^3)$.
\begin{itemize}
    \item Assume that $3|S(R_p)$ and let $d=S(R_p)/3$. Then we have
\[
  \tilde{L}_p(s,R)=p^{3ds}\Big(p^{3d(1-2s)}+P^{3d}_{(1^3)}(s)\Big).
  \]
    where
    \begin{align*}
    P^{3d}_{(1^3)}(s)=&\sum_{0<f\leq 2d,\ 2|f}p^{\min(d,f)-\frac{f}{2}+(3d-\frac{3}{2}f)(1-2s)}+\sum_{0< f\leq d,\ 2\nmid f}p^{\frac{f-1}{2}+(3d-\frac{3f-1}{2})(1-2s)}\\
    &+\sum_{0<f\leq 2d}\sum_{\frac{3}{2}f<\alpha \leq f+\min(d,f)}(1+p^{1-s})p^{\min(d,f)-(\alpha-f)+(3d-\alpha)(1-2s)}\\
    &+\sum_{0< f\leq d}\sum_{\frac{3f-1}{2}<\alpha\leq 2f-1}(1+p^{1-s})p^{2f-\alpha-1+(3d-\alpha)(1-2s)}.
    \end{align*}
     \item   Assume that $3|(S(R_p)-1)$ and let $d=(S(R_p)-1)/3$. Then we have
    \begin{align*}
    \tilde{L}_p(s,R)=p^{(3d+1)s}\Big(p^{(3d+1)(1-2s)}+P^{3d+1}_{(1^3)}(s)\Big),
    \end{align*}
    where
    \begin{align*}
    P^{3d+1}_{(1^3)}(s)=&\sum_{0<f\leq 2d+1,\ 2\nmid f}p^{\min(d,f-1)-\frac{f-1}{2}+(3d+1-\frac{3f-1}{2})(1-2s)}
    +\sum_{0< f\leq d,\ 2\mid f}p^{\frac{f}{2}+(3d+1-\frac{3}{2}f)(1-2s)}\\
    &+\sum_{0<f\leq 2d+1}\sum_{\frac{3f-1}{2}<\alpha \leq f+\min(d,f-1)}(1+p^{1-s})p^{\min(d,f-1)-(\alpha-f)+(3d+1-\alpha)(1-2s)}\\
    &+\sum_{0< f\leq d}\sum_{\frac{3}{2}f<\alpha\leq 2f}(1+p^{1-s})p^{2f-\alpha+(3d+1-\alpha)(1-2s)}.
    \end{align*}
\end{itemize}
\end{proposition}
\begin{proof}
    \textbf{The case that $\mathbf{3|S(R_p)}$.} By the above argument, $R_p$ takes the form $\mathbb{Z}_p[p^dx]$.
    In this case, by \cite[Proposition 4.10 and Corollary 4.11]{CHL2}, the set of overorders of $R_p$ is given by
    \[
    \{\Mfo_{\alpha,f,c}^1\}_{\substack{
    0\leq f\leq 2d,
    \\ \frac{3}{2}f\leq \alpha \leq f+\min(d,f),\\
    c\in  \mathbb{Z}_p/(p^{\min(d,f)-(\alpha-f)})}}\sqcup  \{\Mfo_{\alpha,f,c}^2\}_{\substack{
    1\leq f \leq d,\\
    \frac{3f-1}{2}\leq \alpha\leq 2f-1,\\
    c\in \mathbb{Z}_p/(p^{2f-\alpha-1})}}  
    \]where
    $
\Mfo_{\alpha,f,c}^1=\mathbb{Z}_p\langle 1,p^{\alpha-f}(x+cp^{f-\min(d,f)}x^2),p^f x^2\rangle$ and $
\Mfo_{\alpha,f,c}^2=\mathbb{Z}_p\langle1, p^f x,p^{\alpha-f}(cp x+x^2)\rangle$.
Moreover, by \cite[Corollary 4.11.(1) and (4)]{CHL2}, we have $S(R_p)=3d$, $[\mathcal{O}_{\alpha,f,c}^i:R_p]=p^{3d-\alpha}$, and  
$\Mfo_{\alpha,f,c}^i$ is Gorenstein if and only if $\alpha=\frac{3}{2}f$ for $i=1$ and $\alpha=\frac{3f-1}{2}$ for $i=2$.
Since the residue field of $\mathcal{O}_{a,b,c}^i$ is $\mathbb{F}_p$, we have $\frac{\zeta_{\mathcal{O}_{\alpha,f,c}^i}(s)}{\zeta_{E_p}(s)}=1$.
Substituting these values into the definition of $\tilde{L}_p(s,R)$ and summing over the listed overorders immediately yields the desired formula.

\vspace{0.5em}\noindent
\textbf{The case that $\mathbf{3|(S(R_p)-1)}$.}  $R_p$ takes the form $\mathbb{Z}_p[p^dx^2]$.
Using the fact that $\mathbb{Z}_p[p^{2d+1}x]\subset \mathbb{Z}_p[p^dx^2]$, we repeat the previous arguments for $\mathbb{Z}_p[p^{2d+1}x]$, but this time we select the overorders of $\mathbb{Z}_p[p^{2d+1}x]$ which contain $R_p$.
The set of overorders of $\mathbb{Z}_p[p^{2d+1}x]$ is given by
    \[
    \{\Mfo_{\alpha,f,c}^1\}_{\substack{
    0\leq f\leq 4d+2,
    \\ \frac{3}{2}f\leq \alpha \leq f+\min(d,f),\\
    c\in  \mathbb{Z}_p/(p^{\min(2d+1,f)-(\alpha-f)})}}\sqcup  \{\Mfo_{\alpha,f,c}^2\}_{\substack{
    1\leq f \leq 2d+1,\\
    \frac{3f-1}{2}\leq \alpha\leq 2f-1,\\
    c\in \mathbb{Z}_p/(p^{2f-\alpha-1})}},  
    \]
where     $
\Mfo_{\alpha,f,c}^1=\mathbb{Z}_p\langle 1,p^{\alpha-f}(x+cp^{f-\min(2d+1,f)}x^2),p^f x^2\rangle$ and $
\Mfo_{\alpha,f,c}^2=\mathbb{Z}_p\langle1, p^f x,p^{\alpha-f}(cp x+x^2)\rangle$.
Since the order is given by $R_p = \mathbb{Z}_p \langle 1, p^{2d+1} x, p^d x^2 \rangle$, $R_p \subset \mathcal{O}_{\alpha,f,c}^i$ if and only if $p^d x^2 \in \mathcal{O}_{\alpha,f,c}^i$ and $p^{2d+1} x \in \mathcal{O}_{\alpha,f,c}^i$ for $i=1,2$. 
    
Then, $\mathcal{O}_{\alpha,f,c}^1$ contains $R_p$ if and only if $f \le d$ and $\alpha-f \le 2d+1$. 
Indeed, the first condition directly follows from the condition $p^d x^2 \in \mathcal{O}_{\alpha,f,c}^i$.
For the second condition, we express the element $p^{2d+1} x$ as follows:
\[
p^{2d+1} x = p^{2d+1-(\alpha-f)} \cdot p^{\alpha-f}(x+cx^2) - cp^{2d+1}x^2.
\]
Under the assumption $f \le d$, we have $p^d x^2 \in \mathcal{O}_{\alpha,f,c}^i$, which implies that $cp^{2d+1}x^2 \in \mathcal{O}_{\alpha,f,c}^i$. Therefore, $p^{2d+1} x$ belongs to $\mathcal{O}_{\alpha,f,c}^i$ if and only if $2d+1-(\alpha-f) \ge 0$. 

On the other hand, $\mathcal{O}_{\alpha,f,c}^2$ contains $R_p$ if and only if $f\leq 2d+1$, $\alpha-f\leq d$, and $cp^{d+1}\in (p^f)$.
Indeed, the first condition directly follows from the condition $p^{2d+1}x\in \mathcal{O}_{\alpha,f,c}^2$.
For the second condition we express the element $p^{d}x^2$ as follows:
\[
p^d x^2=p^{d-(\alpha-f)}\cdot p^{\alpha-f}(cpx+x^2)-cp^{d+1}x.
\]
Therefore, $p^{d}x^2$ belongs to $\mathcal{O}_{\alpha,f,c}^2$ if and only if $d-(\alpha-f)\geq 0$ and $cp^{d+1}\in (p^f)$.

In conclusion, the set of overorders of $R_p$ is given by
 \[
    \{\Mfo_{\alpha,f,c}^1\}_{\substack{
    0\leq f\leq d,
    \\ \frac{3}{2}f\leq \alpha \leq 2f,\\
    c\in  \mathbb{Z}_p/(p^{2f-\alpha})}}\sqcup  \{\Mfo_{\alpha,f,c}^2\}_{\substack{
    1\leq f \leq 2d+1,\\
    \frac{3f-1}{2}\leq \alpha\leq f+\min(d,f-1),\\
    c\in p^{f-1-\min(d,f-1)}\mathbb{Z}_p/(p^{2f-\alpha-1})}}.  
    \]
    Here, $\Mfo_{\alpha,f,c}^i$ is Gorenstein if and only if $\alpha=\frac{3}{2}f$ for $i=1$ and $\alpha=\frac{3f-1}{2}$ for $i=2$.
    By substituting the values $S(R_p)=3d+1$, $[\mathcal{O}_{\alpha,f,c}^i:R_p]=p^{3d+1-\alpha}$, and $\frac{\zeta_{\mathcal{O}_{\alpha,f,c}^i}(s)}{\zeta_{E_p}(s)}=1$ in the definition of $\tilde{L}_p(s,R)$, we have the desired formula.
\end{proof}

\begin{theorem}\label{thm:reformulation_for_ram}
Suppose that $\sigma_p(E)=(1^3)$. 
\begin{itemize}
    \item 
        Assume that $3|S(R_p)$ and let $d=S(R_p)/3$. Then we have
        \[
        \tilde{L}_p(s,R)=
        p^{3ds}\Bigg(p^{3d(1-2s)}+\sum_{d'=1}^d p^{3(d-d')(1-2s)}\Big(
        \sum_{i=0}^{d'}p^{(2-3s)i} + p^{1-2s}\sum_{i=0}^{d'-1}p^{(2-3s)i}+ p^{2-4s}\sum_{i=0}^{d'-1}p^{(2-3s)i}\Big)
        \Bigg).
        \]
    \item
        Assume that $3|(S(R_p)-1)$ and let $d=(S(R_p)-1)/3$. Then we have
        \begin{align*}
            &\tilde{L}_p(s,R)\\&=
        p^{(3d+1)s}\Bigg(p^{(3d+1)(1-2s)}+p^{3d(1-2s)}+\sum_{d'=1}^d p^{3(d-d')(1-2s)}\Big(\sum_{i=0}^{d'}p^{(2-3s)i}+p^{1-2s}\sum_{i=0}^{d'}p^{(2-3s)i}+p^{2-4s}\sum_{i=0}^{d'-1}p^{(2-3s)i}\Big)\Bigg).
        \end{align*}
\end{itemize}
\end{theorem}
\begin{proof}
\textbf{The case that $\mathbf{3|S(R_p)}$.}
 By Proposition \ref{prop:first formula for ram}, it suffices to show that
    \begin{equation}\label{eq:comparison_totram_3d}
    P^{3d}_{(1^3)}(s)=\sum_{d'=1}^d p^{3(d-d')(1-2s)}\Big(\sum_{i=0}^{d'}p^{(2-3s)i}+p^{1-2s}\sum_{i=0}^{d'-1}p^{(2-3s)i}+p^{2-4s}\sum_{i=0}^{d'-1}p^{(2-3s)i}\Big).
    \end{equation}
    As in the proof of the case $\sigma_p(E)=(3)$, we first rewrite each term of $P^{3d}_{(1^3)}(s)$ in the form $p^{A(1-2s)+B(2-3s)}$. Then, by reorganizing the indices of summations, we have
    \begin{equation}\label{eq:ram form AB}
    \begin{aligned}
    P^{3d}_{(1^3)}(s)=&\sum_{0<f\leq d,\ 2|f}p^{(3d-3f)(1-2s)+f(2-3s)}+\sum_{0\leq f<d,\ 2|f}p^{f(2-3s)}+\sum_{0<f\leq d,\ 2\nmid f}p^{(3d-3f+2)(1-2s)+(f-1)(2-3s)}\\
    &+\sum_{0\leq \alpha<\frac{d}{2}}\sum_{2\alpha<f\leq d}\Big(p^{(3d-2f-2\alpha)(1-2s)+2\alpha(2-3s)}+p^{(3d-2f-2\alpha-1)(1-2s)+(2\alpha+1)(2-3s)}\Big)\\
    &+\sum_{0\leq \alpha<\frac{d}{2}}\sum_{2d<f<3d-2\alpha}\Big(p^{(3d-f-2\alpha)(1-2s)+2\alpha(2-3s)}+p^{(3d-f-2\alpha-1)(1-2s)+(2\alpha+1)(2-3s)}\Big)\\
    &+\sum_{0\leq \alpha<\frac{d-1}{2}}\sum_{\ 2\alpha+2\leq f\leq d}\Big(p^{(3d-2f-2\alpha+1)(1-2s)+2\alpha(2-3s)}+p^{(3d-2f-2\alpha)(1-2s)+(2\alpha+1)(2-3s)}\Big).
    \end{aligned}
    \end{equation}
    We represent the term $p^{A(1-2s)+B(2-3s)}$ by the pair of numbers $(3d-A,B)$.
    Then it suffices to show that the corresponding multisets of pairs coincide since every monomial $p^{X+Ys}$ admits a unique representation in the form $p^{A(1-2s)+B(2-3s)}$.
    
    The three summations in the first line of \eqref{eq:ram form AB} are represented, for fixed even $i$ with $0\leq i\leq d$, by
    \[
    \left\{\begin{array}{l l}
    \{(1,0),(3d,0)\} &\text{if $i=0$;}\\
    \{(3i,i),(3i+1,i),(3d,i)\} &\text{if $0<i<d$;}\\
    \{(3d,d)\} &\text{if $i=d$.}
    \end{array}\right.
    \]
    The remaining three double summations in \eqref{eq:ram form AB} are represented, for fixed $i$ with $0\leq i\leq d$, by
    \[
    \left\{\begin{array}{l l}
    \{(3i+2,i),(3i+3,i),\dots,(3d-1,i)\} &\text{if $i$ is even;}\\
    \{(3i,i),(3i+1,i),\dots,(3d,i)\} &\text{if $i$ is odd.}
    \end{array}\right.
    \]
    Taking the union of these contributions (for even $i$, the boundary terms supply $(3i,i)$, $(3i+1,i)$ and $(3d,i)$, which fill in the gaps of the double-summation contribution), $P^{3d}_{(1^3)}(s)$ is the summation whose terms correspond to the set of pairs
    \[
    \bigcup_{0\leq i\leq d}\{(x,i):\max(1,3i)\leq x\leq 3d\}.
    \]
    On the other hand, the right-hand side of (\ref{eq:comparison_totram_3d}) corresponds to the set of pairs
    \begin{align*}
    &\bigcup_{d'=1}^{d}\Big(\{(3d',i)\}_{0\leq i\leq d'}\cup\{(3d'-1,i)\}_{0\leq i\leq d'-1}\cup\{(3d'-2,i)\}_{0\leq i\leq d'-1}\Big)\\
    &=\bigcup_{0\leq i\leq d}\{(x,i):\max(1,3i)\leq x\leq 3d\},
    \end{align*}
    where the equality is obtained by collecting, for each fixed $i$, the three residue classes modulo $3$ throughout the range $3i\leq x\leq 3d$ (the value $x=0$ being omitted). The two sets of pairs coincide, which proves (\ref{eq:comparison_totram_3d}).

\vspace{0.5em}\noindent
    \textbf{The case that $\mathbf{S(R_p)=3d+1}$.} By Proposition \ref{prop:first formula for ram}, it suffices to show that
    \begin{equation}\label{eq:comparison_totram_3d1}
    P^{3d+1}_{(1^3)}(s)=p^{3d(1-2s)}+\sum_{d'=1}^d p^{3(d-d')(1-2s)}\Big(\sum_{i=0}^{d'}p^{(2-3s)i}+p^{1-2s}\sum_{i=0}^{d'}p^{(2-3s)i}+p^{2-4s}\sum_{i=0}^{d'-1}p^{(2-3s)i}\Big).
    \end{equation}
    As before, we rewrite each term of $P^{3d+1}_{(1^3)}(s)$ in the form $p^{A(1-2s)+B(2-3s)}$ and reorganize the indices of summations to obtain
    \begin{align*}
    P^{3d+1}_{(1^3)}(s)=&\sum_{0<f\leq d+1,\ 2\nmid f}p^{(3d+3-3f)(1-2s)+(f-1)(2-3s)}+\sum_{0\leq f<d,\ 2|f}p^{f(2-3s)}+\sum_{0<f\leq d,\ 2|f}p^{(3d+1-3f)(1-2s)+f(2-3s)}\\
    &+\sum_{0\leq \alpha<\frac{d}{2}}\sum_{2\alpha+2\leq f\leq d+1}\Big(p^{(3d+2-2f-2\alpha)(1-2s)+2\alpha(2-3s)}+p^{(3d+1-2f-2\alpha)(1-2s)+(2\alpha+1)(2-3s)}\Big)\\
    &+\sum_{0\leq \alpha<\frac{d}{2}}\sum_{2d+1<f\leq 3d-2\alpha}\Big(p^{(3d+1-f-2\alpha)(1-2s)+2\alpha(2-3s)}+p^{(3d-f-2\alpha)(1-2s)+(2\alpha+1)(2-3s)}\Big)\\
    &+\sum_{0\leq \alpha<\frac{d}{2}}\sum_{\ 2\alpha+1\leq f\leq d}\Big(p^{(3d+1-2f-2\alpha)(1-2s)+2\alpha(2-3s)}+p^{(3d-2f-2\alpha)(1-2s)+(2\alpha+1)(2-3s)}\Big).
    \end{align*}
    We represent the term $p^{A(1-2s)+B(2-3s)}$ by the pair of numbers $(3d+1-A,B)$. The three summations in the first line are represented, for fixed even $i$ with $0\leq i\leq d$, by
    \[
    \left\{\begin{array}{l l}
    \{(1,0),(3d+1,0)\} &\text{if $i=0$;}\\
    \{(3i,i),(3i+1,i),(3d+1,i)\} &\text{if $0<i\leq d$,}
    \end{array}\right.
    \]
    while the remaining three double summations are represented, for fixed $i$ with $0\leq i\leq d$, by
    \[
    \left\{\begin{array}{l l}
    \{(3i+2,i),(3i+3,i),\dots,(3d,i)\} &\text{if $i$ is even;}\\
    \{(3i,i),(3i+1,i),\dots,(3d+1,i)\} &\text{if $i$ is odd.}
    \end{array}\right.
    \]
    Taking the union of these contributions, $P^{3d+1}_{(1^3)}(s)$ is the summation whose terms correspond to the set of pairs
    \[
    \bigcup_{0\leq i\leq d}\{(x,i):\max(1,3i)\leq x\leq 3d+1\}.
    \]
    On the other hand, the right-hand side of (\ref{eq:comparison_totram_3d1}) corresponds to the set of pairs
    \begin{align*}
    &\{(1,0)\}\cup\bigcup_{d'=1}^{d}\Big(\{(3d'+1,i)\}_{0\leq i\leq d'}\cup\{(3d',i)\}_{0\leq i\leq d'}\cup\{(3d'-1,i)\}_{0\leq i\leq d'-1}\Big)\\
    &=\bigcup_{0\leq i\leq d}\{(x,i):\max(1,3i)\leq x\leq 3d+1\},
    \end{align*}
    where the leading pair $(1,0)$ comes from the term $p^{3d(1-2s)}$, and the equality is again obtained by collecting, for each fixed $i$, the three residue classes modulo $3$ throughout the range $3i\leq x\leq 3d+1$. The two sets of pairs coincide, which proves (\ref{eq:comparison_totram_3d1}).
\end{proof}
\subsection{The case where $E$ splits at $p$}
\subsubsection{The case where $\sigma_p(E)=(1\ 2)$}\label{sec:reformulation for unramlin}
Recall the isomorphism $E_p \cong \mathbb{Q}_p \times E_p'$, where $E_p'$ is an unramified quadratic extension over $\mathbb{Q}_p$. 
Under this identification, the maximal order is given by $\mathcal{O}_{E_p} = \mathbb{Z}_p \times \mathcal{O}_{E'_p}$. 
We choose $x\in \Mfo_{E_p}^\times$ such that $\Mfo_{E'_p}=\mathbb{Z}_p[x]$.
Then, by \cite[Proposition 5.1]{CHL2}, any order in $E_p$ is explicitly written as
\[
\mathcal{O}_{a,b,c} = \mathbb{Z}_p\langle(1,1), (0,p^a), (0,p^b x+c)\rangle
\]
for some integers $a, b \ge 0$ and $c \in \mathbb{Z}_p$ satisfying the condition $a \le 2\min(b, \ord(c))$.
We emphasize that the integers $a$ and $b$ are intrinsic invariants of a given order $\mathcal{O}_{a,b,c}$, uniquely determined regardless of the choice of generators. 
Explicitly, if the conductor $f(\mathcal{O}_{a,b,c})$ of an order $\mathcal{O}_{a,b,c}$ in $\mathcal{O}_{E_p}$ is given by $(p^{f_1}) \times (p^{f_2})$, then these parameters are recovered by the relations $a = f_1$ and $b = S(\mathcal{O}_{a,b,c}) - f_1$ according to the proof of \cite[Theorem 5.2]{CHL2}.
Consequently, we use the following notation to denote these invariants: \begin{equation}\label{eq:invariants of unramlin}
a(\mathcal{O}_{a,b,c})=a\text{ and }b(\mathcal{O}_{a,b,c})=b.
\end{equation}
Note that $a(R_p)$ is even and $a(R_p)\leq 2b(R_p)$ by \cite[Corollary 5.3]{CHL2} since $R_p$ is Gorenstein.

\begin{proposition}\label{prop:first formula for unramlin}
Suppose that $\sigma_p(E)=(1\ 2)$. Let $a=a(R_p)$ and $b=b(R_p)$ (see \eqref{eq:invariants of unramlin}). Then we have
\begin{equation}\label{eq:unramlin_eq}
\begin{aligned}
\tilde{L}_p(s,R)=&p^{(a+b)s}\Big(p^{(a+b)(1-2s)}+(1+p^{-s})\sum_{0<b'\leq b}p^{(a+b-b')(1-2s)}\Big)+p^{(a+b)s}(1-p^{-2s})P^{a,b}_{(1\ 2)}(s),
         \end{aligned}
         \end{equation}
where
\begin{align*}    
        P^{a,b}_{(1\ 2)}(s)=&\sum_{\substack{0 < a' \leq \frac{a}{2}\\ 2 \mid a'}} \Big(p^{\frac{a'}{2}+(a+b-\frac{3}{2}a')(1-2s)}+\sum_{\frac{a'}{2}< b'\leq b-\frac{a'}{2}}(p^{\frac{a'}{2}}-p^{\frac{a'}{2}-1})p^{(a+b-a'-b')(1-2s)}\Big)\\&+\sum_{\substack{\frac{a}{2} < a' \leq a \\ 2 \mid a'}} p^{\frac{a}{2}-\frac{a'}{2}+(\frac{3a}{2}-\frac{3a'}{2})(1-2s)}
         \\&
         +(1+p^{1-s})\Bigg(\sum_{0<a'\leq \frac{a}{2},\ 2|a'}\Big(\sum_{\frac{a'}{2}< b'\leq b-\frac{a'}{2}}
        p^{\frac{a'}{2}-1+(a+b-a'-b')(1-2s)}+\sum_{b-\frac{a'}{2}<b' \leq b}p^{b-b'+(a+b-a'-b')(1-2s)}\Big)\\
        &+\sum_{\substack{0<a'\leq \frac{a}{2}\\ 2\nmid a'}}\sum_{\frac{a'}{2}\leq b'\leq b}
        p^{\min \left(\lfloor \frac{a'}{2} \rfloor, b-b' \right)+(a+b-a'-b')(1-2s)} + \sum_{\frac{a}{2}<a'\leq a}\sum_{ \frac{a'}{2}+b-\frac{a}{2}< b' \leq b}p^{b-b'+(a+b-a'-b')(1-2s)}\Bigg).
\end{align*} 
\end{proposition}
\begin{proof}
    By \cite[Proposition 5.1.(2)]{CHL2}, any overorder of $R_p=\Mfo_{a,b,c}$ is of the form $\Mfo_{a',b',c'}$.
    Moreover, \cite[Proposition 5.1.(4)]{CHL2} yields that $S(\mathcal{O}_{a,b,c})=a+b$, and $[\mathcal{O}_{a',b',c'}:R_p]=p^{a+b-a'-b'}$.
    If $a'=b'=0$ (so that $\mathcal{O}_{a',b',c'}=\mathcal{O}_{E_p}$), then maximal ideals of $\mathcal{O}_{a',b',c'}$ are $(p)\times \mathcal{O}_{E_p'}$ and $\Z_p\times (p)$, and thus $\zeta_{\mathcal{O}_{a',b',c'}}(s)=\frac{1}{(1-p^{-s})(1-p^{-2s})}$.
    If $a'=0$ and $b'\neq 0$ (so that $\mathcal{O}_{a',b',c'}= \mathbb{Z}_p\times \mathbb{Z}_p[p^{b'}x]$), then maximal ideals of $\mathcal{O}_{a',b',c'}$ are $(p)\times \mathbb{Z}_p[p^{b'}x]$ and $\Z_p\times (p,p^{b'}x)$, and thus $\zeta_{\mathcal{O}_{a',b',c'}}(s)=\frac{1}{(1-p^{-s})^2}$.
    Otherwise, $\mathcal{O}_{a',b',c'}$ is a local ring with the residue field $\mathbb{F}_p$.
    Therefore we have
    \[\frac{\zeta_{\mathcal{O}_{a',b',c'}}(s)}{\zeta_{E_p}(s)}=\left\{
    \begin{array}{l l}
        1 & \text{if $a'=b'=0$}; \\
        1+p^{-s} &\text{if $a'=0$ and $b'\neq 0$};\\
        1-p^{-2s}& \text{otherwise}.
    \end{array}\right.
    \]
    Then, the enumeration of overorders and Gorenstein overorders $\mathcal{O}_{a',b',c'}$ of $\mathcal{O}_{a,b,c}$ for fixed $a'$ and $b'$ in the proof of \cite[Theorem 5.2]{CHL2} yields that 
    \begin{align*}
        \tilde{L}_p(s,R)=&p^{(a+b)s}\Big(p^{(a+b)(1-2s)}+(1+p^{-s})\sum_{0<b'\leq b}p^{(a+b-b')(1-2s)}\Big)\\
        &+p^{(a+b)s}(1-p^{-2s})\Bigg((1+p^{1-s})\sum_{0 \leq b' \leq b} \Big( \sum_{\substack{0 < a' \leq \min(a,\ord(c)) \\ a' \leq 2b' }} p^{\min \left(\lfloor \frac{a'}{2} \rfloor, b-b' \right)+(a+b-a'-b')(1-2s)}\\ &+ \sum_{\substack{\ord(c) < a' \leq a \\ a' \leq 2b' \\ \lceil \frac{a'}{2} \rceil - b' \leq \ord(c) - b}} p^{b-b'+(a+b-a'-b')(1-2s)} \Big)  -p^{1-s}\Big(\sum_{\substack{0 < a' \leq \min(a, \ord(c)) \\ 2 \mid a'}} p^{\min \left( \frac{a'}{2}, b-\frac{a'}{2} \right)+(a+b-\frac{3}{2}a')(1-2s)}
        \\&+ \sum_{\substack{\ord(c) < a' \leq a \\ 2 \mid a' \\ b \leq \ord(c)}} p^{b-\frac{a'}{2}+(a+b-\frac{3a'}{2})(1-2s)}+ \sum_{0 < b' \leq b} \sum_{\substack{0 < a' \leq \min(a,\ord(c)) \\ a' < 2b' \\ 2 \mid a' \\ \frac{a'}{2} \leq b-b'}} \left(p^{\frac{a'}{2}} - p^{\frac{a'}{2}-1} \right)p^{(a+b-a'-b')(1-2s)} \\
        &+ \sum_{\substack{\ord(c) < a' \leq a \\ 2 \mid a' \\ \ord(c) < b}} p^{\ord(c)-\frac{a'}{2}+(a+\ord(c)-\frac{3a'}{2})(1-2s)}\Big)\Bigg).  
    \end{align*}
    Here, we apply the following reformulation to a component of $\tilde{L}_p(s,R)$:
    \begin{align*}
     &p^{S(R_p)s}(1-p^{-2s})
    \sum_{\substack{R_p\subset\mathcal{O}_p\subset \mathcal{O}_{E_p}\\ \zeta_{\mathcal{O}_p}(s)/\zeta_{E_p}(s)=1-p^{-2s}}}h_{\mathcal{O}_p}(s)[\mathcal{O}_p:R_p]^{1-2s} 
    \\&=
         p^{S(R_p)s}(1-p^{-2s})\Big(
    \sum_{\substack{R_p\subset\mathcal{O}_p\subset \mathcal{O}_{E_p}\\ \zeta_{\mathcal{O}_p}(s)/\zeta_{E_p}(s)=1-p^{-2s}}}(1+p^{1-s})[\mathcal{O}_p:R_p]^{1-2s} - \sum_{\substack{R_p\subset\mathcal{O}_p\subset \mathcal{O}_{E_p}\\\zeta_{\mathcal{O}_p}(s)/\zeta_{E_p}(s)=1-p^{-2s}\\ \mathcal{O}_p :\ Gorenstein}}p^{1-s}[\mathcal{O}_p:R_p]^{1-2s}
    \Big).
    \end{align*}
    By the same argument in the proof of \cite[Theorem 5.2]{CHL2}, replacing $\ord(c)$ with $S(R_p)-f_2$ leaves the sum invariant, where $(p^{f_1})\times (p^{f_2})\subset \mathbb{Z}_p\times \mathcal{O}_{E_p'} $ is the conductor of an order $R_p$ in $\mathcal{O}_{E_p}$.
    Since $R_p$ is Gorenstein, \cite[Corollary 5.3]{CHL2} yields that \[S(R_p)-f_2=\frac{f_1}{2}=\frac{a}{2}.\] 
    Therefore, we replace $\ord(c)$ with $\frac{a}{2}$. 
    Interchanging the order of the double sums and applying the identity $-p^{1-s} = 1 - (1+p^{1-s})$ yields the desired formula.
\end{proof}
\begin{theorem}\label{thm:reformulation_for_unramlin}
Suppose that $\sigma_p(E)=(1\ 2)$. Let $a=a(R_p)$ and $b=b(R_p)$ (see \eqref{eq:invariants of unramlin}). Then we have
    \begin{align*}
    \tilde{L}_p(s,R)=
&p^{(a+b)(1-s)}+(1+p^{-s})\Big(
        p^{(a+b-1)(1-s)+s}\sum_{i=0}^{\frac{a}{2}-1}p^{(3s-1)i}+
        p^{(b+\frac{a}{2}-1)(1-s)+as}\sum_{i=0}^{b-\frac{a}{2}-1}p^{(2s-1)i}\Big)\\
        &+p^{(a+b)s}(1-p^{-2s})\Bigg(\sum_{d'=0}^{b-\frac{a}{2}-1}p^{(1-2s)d'}\sum_{i=0}^{\frac{a}{2}}p^{(2-3s)i}
        \\&+p^{(b-\frac{a}{2})(1-2s)}\sum_{d'=1}^{\frac{a}{2}}p^{(3-6s)(\frac{a}{2}-d')}\Big(\sum_{i=0}^{d'}p^{(2-3s)i}+p^{1-2s}\sum_{i=0}^{d'-1}p^{(2-3s)i}+p^{2-4s}\sum_{i=0}^{d'-2}p^{(2-3s)i}\Big)\Bigg).
    \end{align*}
\end{theorem}
\begin{proof}
Plugging in $a=0$ in the expression (\ref{eq:unramlin_eq}) in Proposition \ref{prop:first formula for unramlin}, we have
\begin{align*}
    \tilde{L}_p(s,R)=&p^{bs}\Big(p^{b(1-2s)}+(1+p^{-s})\sum_{0<b'\leq b}p^{(b-b')(1-2s)}\Big)\\
    =& p^{b(1-s)}+(1+p^{-s})p^{bs}\sum_{b'=0}^{b-1}p^{(1-2s)b'}\\
    =& p^{b(1-s)}+(1+p^{-s})\Big( p^{(b-1)(1-s)}\sum_{i=0}^{b-1}p^{(2s-1)i}+(1-p^{-s})p^{bs}\sum_{d'=0}^{b-1}p^{(1-2s)d'}\Big).
\end{align*}
This proves the theorem for $a=0$.

In the case where $a\geq 2$, we claim that the expression (\ref{eq:unramlin_eq}) for $\tilde{L}_p(s,R)$ equals the following one, which is the assertion of the theorem after factoring out $p^{(a+b)s}$:
\begin{equation}\label{eq:reformulation of state in thm of unramlin}
\begin{aligned}
    &p^{(a+b)s}\Bigg(p^{(a+b)(1-2s)}+(1+p^{-s})\sum_{1\leq i\leq b}p^{(a+b-i)(1-2s)}\Bigg)\\
        &+p^{(a+b)s}(1-p^{-2s})\Bigg(\sum_{d'=0}^{b-\frac{a}{2}-1}p^{(1-2s)d'}\sum_{i=0}^{\frac{a}{2}}p^{(2-3s)i}
        +\sum_{d'=0}^{a/2-1}p^{(a+b-1-d')(1-2s)+d's}\sum_{i=0}^{d'-1}p^{-is}\\
        &+\sum_{d'=a/2}^{b-1}p^{(a+b-1-d')(1-2s)+(a/2-1)s}\sum_{i=0}^{a/2-2}p^{-is}
        \\&+p^{(b-\frac{a}{2})(1-2s)}\sum_{d'=1}^{\frac{a}{2}}p^{(3-6s)(\frac{a}{2}-d')}(\sum_{i=0}^{d'}p^{(2-3s)i}+p^{1-2s}\sum_{i=0}^{d'-1}p^{(2-3s)i}+p^{2-4s}\sum_{i=0}^{d'-2}p^{(2-3s)i})\Bigg)
\end{aligned}
\end{equation}
The first line is identical with the first bracketed term of $\tilde{L}_p(s,R)$ in (\ref{eq:unramlin_eq}), so it suffices to prove that the sum multiplied by $p^{(a+b)s}(1-p^{-2s})$ and $P_{(1\ 2)}^{a,b}(s)$ coincide. We write $P_{(1\ 2)}^{a,b}(s)$ as follows
\begin{equation}\label{eq:LHS of unramlin}
P_{(1\ 2)}^{a,b}(s)=\mathcal{A}(s)+(1+p^{1-s})\mathcal{C}(s),
\end{equation}
where
\begin{align*}
\mathcal{A}(s)=&\sum_{\substack{0 < a' \leq \frac{a}{2}\\ 2 \mid a'}} \Big(p^{\frac{a'}{2}+(a+b-\frac{3}{2}a')(1-2s)}+\sum_{\frac{a'}{2}< b'\leq b-\frac{a'}{2}}(p^{\frac{a'}{2}}-p^{\frac{a'}{2}-1})p^{(a+b-a'-b')(1-2s)}\Big)\\&+\sum_{\substack{\frac{a}{2} < a' \leq a \\ 2 \mid a'}} p^{\frac{a}{2}-\frac{a'}{2}+(\frac{3a}{2}-\frac{3a'}{2})(1-2s)},\\
\mathcal{C}(s)=&\sum_{\substack{0<a'\leq \frac{a}{2}\\ 2\mid a'}}\Big(\sum_{\frac{a'}{2}< b'\leq b-\frac{a'}{2}}p^{\frac{a'}{2}-1+(a+b-a'-b')(1-2s)}+\sum_{b-\frac{a'}{2}<b' \leq b}p^{b-b'+(a+b-a'-b')(1-2s)}\Big)\\
&+\sum_{\substack{0<a'\leq \frac{a}{2}\\ 2\nmid a'}}\sum_{\frac{a'}{2}\leq b'\leq b-\frac{a'}{2}}p^{\lfloor \frac{a'}{2}\rfloor+(a+b-a'-b')(1-2s)}+\sum_{\substack{0<a'\leq \frac{a}{2}\\ 2\nmid a'}}\sum_{b-\frac{a'}{2}< b'\leq b}p^{ b-b'+(a+b-a'-b')(1-2s)}
\\&+\sum_{\frac{a}{2}<a'\leq a}\sum_{\frac{a'}{2}+b-\frac{a}{2}<b'\leq b}p^{b-b'+(a+b-a'-b')(1-2s)},
\end{align*}
and we denote by $Q_{(1\ 2)}^{a,b}(s)$ the bracket multiplied by $p^{(a+b)s}(1-p^{-2s})$ in \eqref{eq:reformulation of state in thm of unramlin},
\begin{equation}\label{eq:rhs_of_unramsplit}
\begin{aligned}
Q_{(1\ 2)}^{a,b}(s)=&\sum_{d'=0}^{b-\frac{a}{2}-1}p^{(1-2s)d'}\sum_{i=0}^{\frac{a}{2}}p^{(2-3s)i}&\text{(i)}\\
&+\sum_{d'=0}^{a/2-1}p^{(a+b-1-d')(1-2s)+d's}\sum_{i=0}^{d'-1}p^{-is}
+\sum_{d'=a/2}^{b-1}p^{(a+b-1-d')(1-2s)+(a/2-1)s}\sum_{i=0}^{a/2-2}p^{-is} &\text{(ii)}\\
&+p^{(b-\frac{a}{2})(1-2s)}\sum_{d'=1}^{\frac{a}{2}}p^{(3-6s)(\frac{a}{2}-d')}\Big(\sum_{i=0}^{d'}p^{(2-3s)i}+p^{1-2s}\sum_{i=0}^{d'-1}p^{(2-3s)i}+p^{2-4s}\sum_{i=0}^{d'-2}p^{(2-3s)i}\Big).&\text{(iii)}
\end{aligned}
\end{equation}
Thus it suffices to show that
\begin{equation}\label{eq:comparison_splitunram}
P_{(1\ 2)}^{a,b}(s)=Q_{(1\ 2)}^{a,b}(s).
\end{equation}

Throughout we abbreviate $a=2m$ and $b=m+k$, so that $a$ even and $a\leq 2b$ amount to $m\geq 1$ and $k\geq 0$, and $a+b=3m+k=:T$.
As in the proof of the case $\sigma_p(E)=(3)$ and $(1^3)$, we rewrite every monomial occurring on either side of \eqref{eq:comparison_splitunram} in the form $p^{A(1-2s)+B(2-3s)}$, and we record it by the pair $(A,B)$. 
Since a monomial $p^{X+Ys}$ admits a unique representation in the form $p^{A(1-2s)+B(2-3s)}$, it suffices to show that the corresponding multisets of pairs obtained from $P_{(1\ 2)}^{a,b}(s)$ and $Q_{(1\ 2)}^{a,b}(s)$ coincide.
Thus we compute both and show they equal the same multiset
\begin{equation}\label{eq:Mset_splitunram}
\begin{split}
\mathcal{M}=
\Big(\{(A,0): 0\leq A \leq T-2\}\sqcup
\bigsqcup_{1\leq i\leq m}&\{(A,i): 0\leq A\leq T-3i\}\Big) \\
&\sqcup \Big(\bigsqcup_{1\leq i\leq m-1}\{(A,i): 2(m-i)\leq A\leq T-3i-2\}\Big),
\end{split}
\end{equation}
the first union recording the monomials with multiplicity one and the second recording the central segment of each row $1\leq i\leq m-1$ that occurs a \emph{second} time. 

\emph{$Q_{(1\ 2)}^{a,b}(s)$ equals $\mathcal{M}$.} 
The first family (i) $\sum_{d'=0}^{b-\frac{a}{2}-1}p^{(1-2s)d'}\sum_{i=0}^{\frac{a}{2}}p^{(2-3s)i}$ of $Q_{(1\ 2)}^{a,b}(s)$ clearly corresponds to \[\bigsqcup_{0\leq i \leq m}\{(A,i): 0\leq A \leq k-1\}.\] 
For the family (ii) of $Q_{(1\ 2)}^{a,b}(s)$, we write $p^{-is}= p^{-i(2-3s)+2i(1-2s)}$. 
After the substitution $d'$ to $A$ and reindexing, the family (ii) of $Q_{(1\ 2)}^{a,b}(s)$ becomes
\[\sum_{1\leq i\leq m-1}\ \sum_{2(m-i)\leq A\leq T-3i-1}p^{A(1-2s)+i(2-3s)},\]
so that corresponds to $\bigsqcup_{1\leq i \leq m-1}\{(A,i): 2(m-i)\leq A \leq T-3i-1\}$.
On the other hand, according to the proof of Theorem \ref{thm:reformulation_for_unram}, the family (iii) of $Q_{(1\ 2)}^{a,b}(s)$ contributes the following set of pairs 
\begin{align*}
        \{(A,0): k\leq A\leq T-2\}\cup\bigcup_{1\leq i\leq m-1}
        \Big(\{(A,i): k\leq A \leq T-3i-2\}\cup (T-3i,i)\Big)
        \cup \{(k,m)\}.
    \end{align*}
One checks that (i)+(iii) only contributes the first union in the right-hand side of \eqref{eq:Mset_splitunram} at most one.
Family (ii) then adds a second copy of each pair $(A,i)$ with $1\leq i \leq m-1$ and $2(m-i)\leq A \leq T-3i-2$.
Indeed, Family (ii) initially produces one additional pair with $A=T-3i-1$, but this unique overhang is cancelled by the corresponding gap in the contribution from Families (i) and (iii).  
Hence the pair-multiset of $Q_{(1\ 2)}^{a,b}(s)$ corresponds to exactly $\mathcal{M}$.



\emph{$P_{(1\ 2)}^{a,b}(s)$ equals $\mathcal{M}$.} In the second summand of $\mathcal{A}(s)$, the negative part $-p^{\frac{a'}2-1}p^{(T-a'-b')(1-2s)}$ (summed over even $a'$ and $\frac{a'}2<b'\leq m+k-\frac{a'}2$) is identical to the first inner sum of $\mathcal{C}(s)$.
Consequently, in $P_{(1\ 2)}^{a,b}(s)=\mathcal{A}(s)+(1+p^{1-s})\mathcal{C}(s)$ the term $-p^{\frac{a'}2-1}(\cdots)$ of $\mathcal A(s)$ cancels that inner sum, leaving
\[
P_{(1\ 2)}^{a,b}(s)=\widehat{\mathcal A}(s)+\widehat{\mathcal C}(s)+p^{1-s}\,\mathcal{C}(s),
\]
where $\widehat{\mathcal A}(s)$ is $\mathcal{A}(s)$ with the $-p^{\frac{a'}2-1}(\cdots)$ removed and $\widehat{\mathcal C}(s)$ is $\mathcal{C}(s)$ with its first inner sum removed.
Moreover, in $\widehat{\mathcal{C}}(s)$ and $\mathcal{C}(s)$, the summations with $0<a' \leq m$ and $m+k-\frac{a'}{2}<b'\leq m+k$ for $2\mid a'$ and $2\nmid a'$ are structurally equivalent under the floor function, we combine them for $0<a'\leq m$:
\[
\sum_{0 < a' \leq m} \sum_{m+k-\lfloor a'/2 \rfloor < b' \leq m+k} p^{(m+k-b')+(T-a'-b')(1-2s)}.
\] 

We decompose $\widehat{\mathcal{A}}(s)$ into the following sums, so that  $\widehat{\mathcal{A}}(s)=\mathcal{A}_1(s)+\mathcal{A}_2(s)+\mathcal{A}_3(s)$:
\begin{gather*}
\mathcal{A}_1(s) = \sum_{\substack{0 < a' \le m \\ 2 \mid a'}} p^{\frac{a'}{2} + (T-\frac{3a'}{2})(1-2s)},\quad
\mathcal{A}_{2}(s) = \sum_{\substack{0 < a' \le m \\ 2 \mid a'}} \sum_{\frac{a'}{2} < b' \le m+k-\frac{a'}{2}} p^{\frac{a'}{2}+(T-a'-b')(1-2s)},
\\
\mathcal{A}_{3}(s) = \sum_{\substack{m < a' \le 2m \\ 2 \mid a'}} p^{m-\frac{a'}{2} + (3m-\frac{3a'}{2})(1-2s)}.
\end{gather*}
On the other hand, we define the following terms so that $\mathcal{C}(s)=\mathcal{C}_1(s)+\mathcal{C}_2(s)+\mathcal{C}_3(s)+\mathcal{C}_4(s)$ and $\widehat{C}(s)=\mathcal{C}_2(s)+\mathcal{C}_3(s)+\mathcal{C}_4(s)$:
\begin{align*}
&\mathcal{C}_1(s) = \sum_{\substack{0 < a' \le m \\ 2 \mid a'}} \sum_{\frac{a'}{2} < b' \le m+k-\frac{a'}{2}} p^{\frac{a'}{2}-1 + (T-a'-b')(1-2s)}, 
&\mathcal{C}_2(s) = \sum_{\substack{0 < a' \le m }} \sum_{m+k-\frac{a'}{2} < b' \le m+k} p^{m+k-b' + (T-a'-b')(1-2s)}, \\
&\mathcal{C}_3(s) = \sum_{\substack{0 < a' \le m \\ 2 \nmid a'}} \sum_{\frac{a'}{2} \le b' \le m+k-\frac{a'}{2}} p^{\lfloor \frac{a'}{2}\rfloor + (T-a'-b')(1-2s)}, &\mathcal{C}_4(s) = \sum_{m < a' \le 2m} \sum_{\frac{a'}{2}+k < b' \le m+k} p^{m+k-b' + (T-a'-b')(1-2s)}.
\end{align*}

Passing to the form $p^{A(1-2s)+B(2-3s)}$, 
for any base term $p^{X+Y(1-2s)}$ appearing in $\mathcal{A}_i(s)$ and $\mathcal{C}_j(s)$, equating exponents gives $B=2X$ and $A=Y-3X$. Therefore, $\widehat{\mathcal{A}}(s)+\hat{\mathcal{C}}(s)$ in $P_{(1\ 2)}^{a,b}(s)$ generate pairs $(A,i)$ solely for even $i$.
Conversely, the shifted terms $p^{1-s} \cdot p^{X+Y(1-2s)}$  shift the indices to $B=2X+1$ and $A=Y-3X-1$. Thus, $p^{1-s}\mathcal{C}(s)$ generates pairs $(A,i)$ solely for odd $i$. 
We now analyze the intervals of $A$ in $(A,i)$ generated for fixed $i$.

\vspace{0.5em}
\noindent\textbf{Case 1: $i=0$ or $m$.}\\
In the case that $i=0$, a pair $(A,0)$ arises solely from $\widehat{\mathcal{A}}(s)+\widehat{\mathcal{C}}(s)$. In particular, it arises from $\mathcal{A}_3(s)$, $\mathcal{C}_2(s)$, $\mathcal{C}_3(s)$, and $\mathcal{C}_4(s)$. 
\begin{itemize}
    \item From $\mathcal{A}_3(s)$, for $a'=2m$, we have  $(0,0)$.
    \item From $\mathcal{C}_4(s)$, for $b'=m+k$, we have $(A,0)$ where $A$ ranges over $[1,m-1]$.
    \item From $\mathcal{C}_2(s)$, for $b'=m+k$, we have $(A,0)$ where $A$ ranges over $[m,2m-1]$.
    \item From $\mathcal{C}_3(s)$, for $a'=1$, we have $(A,0)$ where $A$ ranges over $[2m,T-2]$.
\end{itemize}
Taking the union of these intervals gives $[0,T-2]$.
We now consider the case $i=m$. If $m$ is even, pairs $(A,m)$ arise from $\mathcal{A}_1(s)+\mathcal{A}_2(s)$, yielding the interval $[0,T-3m]$ for $A$. 
If $m$ is odd, pairs $(A,m)$ arise from $p^{1-s}(\mathcal{C}_2(s)+\mathcal{C}_3(s))$, which leads to the same interval $[0,T-3m]$ for $A$.


\vspace{0.5em}
\noindent \textbf{Case 2: Even $i$ ($1 \le i \le m-1$).} \\
Pairs $(A,i)$ with even $i$ arise solely from $\widehat{\mathcal{A}}(s)+ \widehat{\mathcal{C}}(s)$. 
\begin{itemize}
    \item From $\mathcal{A}_3(s)$, we have $i=2(m -\frac{a'}{2})=2m-a'$
    The $A$-coordinate is $3m-\frac{3a'}{2} - 3(m -\frac{a'}{2}) = 0$.
    \item 
    From $\mathcal{C}_4(s)$ and $\mathcal{C}_2(s)$, we have $i=2(m+k-b')$. The $A$-coordinate is $2m-a'-i$ with $m<a'<2m-i$ in $\mathcal{C}_4(s)$ and $2m-a'-i$ with $i<a'\leq m$ in $\mathcal{C}_2(s)$. Thus, $A$ covers the integer interval $[1,2m-2i-1]$.
    \item
    From $\mathcal{C}_3(s)$, we have $i=2\cdot \frac{a'-1}{2}=a'-1$ since $a'$ is odd. 
    The $A$-coordinate is $T-i-1-b'-3\cdot \frac{i}{2}$ where $\frac{i}{2}+1\leq b'\leq m+k-\frac{i}{2}-1$, and
    thus $A$ spans the integer interval $[2m-2i,T-3i-2]$.
    
    \item From $\mathcal{A}_1(s)$, we have $i=2\cdot \frac{a'}{2}=a'$.
    The $A$-coordinate is $(T-\frac{3a'}{2}) - 3\cdot \frac{a'}{2} = T-3i$.
    
    \item From $\mathcal{A}_{2}(s)$, we have $i = a'$. The $A$-coordinate is $T-i-b' - 3\cdot \frac{i}{2} = T - b' - \frac{5i}{2}$ with $b'$ ranges from $\frac{i}{2}+1\leq b'\leq m+k-\frac{i}{2}$, and thus $A$ spans the interval $[2m-2i, T-3i-1]$.
\end{itemize}
Taking the multiset union of these intervals gives $[0, T-3i] \sqcup [2(m-i), T-3i-2]$, where the overlapping region $[2(m-i), T-3i-2]$ appears with multiplicity $2$.

\vspace{0.5em}
\noindent \textbf{Case 3: Odd $i$ ($1 \le i \le m-1$).} \\
Pairs $(A,i)$ with odd $i$ arise solely from $p^{1-s}\mathcal{C}(s)$. 
\begin{itemize}
    \item By $\mathbf{Case\ 1}$ and $\mathbf{Case\ 2}$, for an even $i'=i-1$ with $0\leq i' \leq m-1$, the term $\mathcal{C}_2(s)+\mathcal{C}_3(s)+\mathcal{C}_4(s)$ generates pairs $(A',i')$ where $A' \in [1, T-3i'-2]$. 
    Applying the shift $p^{1-s}=p^{(2-3s)-(1-2s)}$, 
    the term $p^{1-s}(\mathcal{C}_2(s)+\mathcal{C}_3(s)+\mathcal{C}_4(s))$ shifts these to pairs $(A,i)=(A'-1,i'+1)$, where $A$ ranges over $[0, T-3i]$.
    
    \item From $p^{1-s}\mathcal{C}_1(s)$, we have $i=(a'-2)+1$.
    The $A$-coordinate is calculated as $(T-a'-b') - 3(\frac{a'}{2}-1) - 1 = T - \frac{5a'}{2} - b' + 2$. As $b'$ ranges from $\frac{i+1}{2}+1$ to $m+k-\frac{i+1}{2}$, $A$ spans the interval $[2m-2i, T-3i-2]$.
\end{itemize}
The multiset union for odd $i$ is therefore $[0, T-3i] \sqcup [2(m-i), T-3i-2]$, matching the even case exactly.

\vspace{0.5em}\noindent
Summing over all $0 \le i \le m$, we conclude that the derived multiset of pairs $(A,B)$ from $P_{(1\ 2)}^{a,b}(s)$ is equivalent to the target multiset $\mathcal{M}$.
\end{proof}
\subsubsection{The case where $\sigma_p(E)=(1\ 1^2)$}\label{sec:reformulation for ramlin}
As in the previous section, we recall the isomorphism $E_p\cong \mathbb{Q}_p\times E_p'$ where $E_p'$ is a ramified quadratic field extension over $\mathbb{Q}_p$. 
Then the maximal order is given by $\mathcal{O}_{E_p}=\mathbb{Z}_p\times \mathcal{O}_{E_p'}$ under this identification.
We choose a uniformizer $x$ of $E'_p$ so that $\Mfo_{E'_p}=\mathbb{Z}_p[x]$.
By \cite[Proposition 5.4.(1)]{CHL2}, any order in $E_p$ is of the form \[\mathcal{O}_{a,b,c}=\mathbb{Z}_p\langle (1,1),(0,p^a),(0,p^bx+c)\rangle\] for some integers $a,b\geq 0$ and $c\in \mathbb{Z}_p$ satisfying the condition $a\leq \min(2b+1,2\ord(c))$.
Here, the integers $a$ and $b$ are intrinsic invariants of a given order $\mathcal{O}_{a,b,c}$, uniquely determined regardless of the choice of generators.
If the conductor $f(\mathcal{O}_{a,b,c})$ of an order $\mathcal{O}_{a,b,c}$ in $\mathcal{O}_{E_p}$ is given by $(p^{f_1})\times (x^{f_2})$, then we have $a=f_1$ and $b=S(\mathcal{O}_{a,b,c})-f_1$ according to the proof of \cite[Theorem 5.5]{CHL2}.
Thus, we define the following notation to denotes this invariants:
\begin{equation}\label{eq:invariants of ramlin}
    a(\mathcal{O}_{a,b,c})=a \text{ and }b(\mathcal{O}_{a,b,c})=b.
\end{equation}
Since $R_p$ is Gorenstein, \cite[Corollary 5.6]{CHL2} yields that $a(R_p)\leq 2b(R_p)$ if $a(R_p)$ is even, and $a(R_p)=2b(R_p)+1$ if $a(R_p)$ is odd.

\begin{proposition}\label{prop:first formula for ramlin}
    Suppose $\sigma_p(E)=(1\ 1^2)$. Let $a=a(R_p)$ and $b=b(R_p)$ (see \eqref{eq:invariants of ramlin}). 
    \begin{itemize}
        \item If $a$ is even, then we have 
    \begin{align*}
    \tilde{L}_p(s,R) =&p^{(a+b)s}\Big(p^{(a+b)(1-2s)}+\sum_{1\leq b'\leq b}p^{(a+b-b')(1-2s)}\Big)+p^{(a+b)s}(1-p^{-s})P_{(1\ 1^2)}^{a,b},
    \end{align*}
    where
    \begin{align*}
             P_{(1\ 1^2)}^{a,b}=&\sum_{\substack{0< a' \leq \frac{a}{2}   \\ 2\nmid a'}}p^{\frac{a'-1}{2}+(a+b-\frac{3a'-1}{2})(1-2s)}
    \\&+\sum_{{\substack{0< a' \leq \frac{a}{2} \\  2\mid a'}}}\sum_{\substack{\frac{a'-1}{2}<b'\leq b-\frac{a'}{2}}}(p^{\frac{a'}{2}}-p^{\frac{a'}{2}-1})p^{(a+b-a'-b')(1-2s)}+\sum_{\substack{\frac{a}{2}<a'\leq a\\ 2\mid a' }}p^{\frac{a}{2}-\frac{a'}{2}+(\frac{3a}{2}-\frac{3a'}{2})(1-2s)}\\&+(1+p^{1-s})\Bigg(\sum_{\substack{0< a'\leq \frac{a}{2}\\ 2\mid a'}}\sum_{\substack{\frac{a'-1}{2}\leq b'\leq b-\frac{a'}{2}}}
    p^{\frac{a'}{2}-1+(a+b-a'-b')(1-2s)}+  \sum_{\substack{0< a'\leq \frac{a}{2}}}\sum_{\substack{b-\frac{a'}{2}< b'\leq b}}p^{b-b'+(a+b-a'-b')(1-2s)}
    \\&+\sum_{\substack{0< a'\leq \frac{a}{2}\\ 2\nmid a'}}\sum_{\substack{\frac{a'+1}{2}\leq b'\leq b-\frac{a'}{2}}}
    p^{\lfloor \frac{a'}{2}\rfloor +(a+b-a'-b')(1-2s)}+\sum_{\substack{
    \frac{a}{2}<a'\leq a}}\sum_{ \frac{a'}{2}+b-\frac{a}{2} < b'\leq b}p^{b-b'+(a+b-a'-b')(1-2s)}\Bigg).
   \end{align*}    
        \item If $a$ is odd, then we have
         \begin{align*}
      \tilde{L}_p(s,R)=
    &p^{(3b+1)s}\Big(p^{(3b+1)(1-2s)}+\sum_{1\leq b'\leq b}p^{((3b+1)-b')(1-2s)}\Big)+p^{(3b+1)s}(1-p^{-s})P_{(1\ 1^2)}^{2b+1,b},
    \end{align*}
    where
    \begin{align*}
    P_{(1\ 1^2)}^{2b+1,b}=&\sum_{\substack{0< a' \leq b+1  \\ 2\nmid a'}}p^{\frac{a'-1}{2}+(3b+1-\frac{3a'-1}{2})(1-2s)}\\&+\sum_{{\substack{0< a' \leq b+1\\  2\mid a'}}}\sum_{\substack{\frac{a'}{2}\leq b'\leq b-\frac{a'}{2}}}(p^{\frac{a'}{2}}-p^{\frac{a'}{2}-1})p^{(3b+1-a'-b')(1-2s)} 
    +\sum_{\substack{b+1<a'\leq 2b+1 \\ 2\nmid a'}}p^{b-\frac{a'-1}{2}+(3b+1-\frac{3a'-1}{2})(1-2s)}\\
    &+(1+p^{1-s})\Bigg(\sum_{\substack{0< a'\leq b+1\\ 2\mid a'}}
    \sum_{\frac{a'}{2}\leq b'\leq b-\frac{a'}{2} }p^{\frac{a'}{2}-1+(3b+1-a'-b')(1-2s)}
    +\sum_{\substack{0< a'\leq b+1}}\sum_{b-\lfloor\frac{a'}{2}\rfloor<b'\leq b}p^{b-b'+(3b+1-a'-b')(1-2s)}\\
    &+\sum_{\substack{0< a'\leq b+1\\ 2\nmid a'}} \sum_{\substack{\frac{a'+1}{2}\leq  b'\leq b-\frac{a'-1}{2}}}p^{\frac{a'-1}{2}+(3b+1-a'-b')(1-2s)}+\sum_{b+1 <a'\leq 2b+1}\sum_{\frac{a'-1}{2} < b'\leq b}p^{b-b'+(3b+1-a'-b')(1-2s)}\Bigg).
    \end{align*}
    \end{itemize}
\end{proposition}
\begin{proof}
    By \cite[Proposition 5.4.(2)]{CHL2}, any overorder of $R_p=\Mfo_{a,b,c}$ is of the form $\Mfo_{a',b',c'}$.
    Moreover, \cite[Proposition 5.4.(4)]{CHL2} yields that $S(\mathcal{O}_{a,b,c})=a+b$ and $[\mathcal{O}_{a',b',c'}:R_p]=p^{a+b-a'-b'}$. 
    If $a'=0$ (so that $\mathcal{O}_{a',b',c'}= \mathbb{Z}_p\times \mathbb{Z}_p[p^{b'}x]$), then maximal ideals of $\mathcal{O}_{a',b',c'}$ are $(p)\times \mathbb{Z}_p[p^{b'}x]$ and $\Z_p\times (p,p^{b'}x)$, and thus $\zeta_{\mathcal{O}_{a',b',c'}}(s)=\frac{1}{(1-p^{-s})^2}$.
    Otherwise, $\mathcal{O}_{a',b',c'}$ is a local ring with the residue field $\mathbb{F}_p$. Therefore we have
    \[\frac{\zeta_{\mathcal{O}_{a',b',c'}}(s)}{\zeta_{E_p}(s)}=\left\{
    \begin{array}{l l}
        1 & \text{if $a'=0$}; \\
        1-p^{-s}& \text{otherwise}.
    \end{array}\right.
    \]
    Then the enumeration of overorders and Gorenstein overorders $\mathcal{O}_{a',b',c'}$ of $\mathcal{O}_{a,b,c}$ for fixed $a'$ and $b'$ in the proof of \cite[Theorem 5.5]{CHL2} yields that 
    \begin{equation}
    \begin{aligned}
        \tilde{L}_p(s,R)=
      &p^{(a+b)s}\Big(p^{(a+b)(1-2s)}+\sum_{0<b'\leq b}p^{(a+b-b')(1-2s)}\Big) \\
      &+p^{(a+b)s}(1-p^{-s})\Bigg((1+p^{1-s})\Big(\sum_{0\leq b'\leq b }
    \sum_{\substack{0< a'\leq \min(a,\ord(c))\\ a'\leq 2b'+1 }}
    p^{\min (\lfloor\frac{a'}{2} \rfloor,b-b')+(a+b-a'-b')(1-2s)}\\
    &+\sum_{\substack{\ord(c)<a'\leq a\\a'\leq 2b'+1\\ \lceil \frac{a'}{2}\rceil  \leq \ord(c)-(b-b')}}p^{b-b'+(a+b-a'-b')(1-2s)}\Big)-p^{1-s}\Big(\sum_{\substack{0< a' \leq a,\\a'\leq \ord(c) \\ 2|(a'-1)}}p^{\min(\frac{a'-1}{2},b-\frac{a'-1}{2})+(a+b-\frac{3a'-1}{2})(1-2s)}
    \\&+\sum_{\substack{\ord(c)<a'\leq a\\ 2|(a'-1)\\b+1\leq \ord(c)}}p^{b-\frac{a'-1}{2}+(a+b-\frac{3a'-1}{2})(1-2s)}
    +\sum_{\substack{0<b'\leq b}}\sum_{{\substack{0< a' \leq \min(a,\ord(c)),\\  a'< 2b'+1,\ 2\mid a'\\ \frac{a'}{2}\leq b-b'}}}(p^{\frac{a'}{2}}-p^{\frac{a'}{2}-1})p^{(a+b-a'-b')(1-2s)}\\
    &+\sum_{\substack{\ord(c)<a'\leq a\\ 2|a'\\ \ord(c)\leq b}}p^{\ord(c)-\frac{a'}{2}+(a+\ord(c)-\frac{3a'}{2})(1-2s)}\Big)\Bigg)
    \end{aligned}
    \end{equation}
    Here, we apply the following reformulation to a component of $\tilde{L}_p(s,R)$:
    \begin{align*}
     &p^{S(R_p)s}(1-p^{-s})
    \sum_{\substack{R_p\subset\mathcal{O}_p\subset \mathcal{O}_{E_p}\\ \zeta_{\mathcal{O}_p}(s)/\zeta_{E_p}(s)=1-p^{-s}}}h_{\mathcal{O}_p}(s)[\mathcal{O}_p:R_p]^{1-2s} 
    \\&=p^{S(R_p)s}(1-p^{-s})\Big(
    \sum_{\substack{R_p\subset\mathcal{O}_p\subset \mathcal{O}_{E_p}\\ \zeta_{\mathcal{O}_p}(s)/\zeta_{E_p}(s)=1-p^{-s}}}(1+p^{1-s})[\mathcal{O}_p:R_p]^{1-2s} - \sum_{\substack{R_p\subset\mathcal{O}_p\subset \mathcal{O}_{E_p}\\ \zeta_{\mathcal{O}_p}(s)/\zeta_{E_p}(s)=1-p^{-s}\\ \mathcal{O}_p :\ Gorenstein}}p^{1-s}[\mathcal{O}_p:R_p]^{1-2s}
    \Big).
    \end{align*}
    By the same argument in the proof of \cite[Theorem 5.5]{CHL2}, replacing $\ord(c)$ with $\lceil \frac{2S(R_p)-f_2}{2}\rceil$ leaves the sum invariant where $(p^{f_1})\times (x^{f_2})$ is the conductor of an order $R_p$ in $\mathcal{O}_{E_p}$.
    Since $R_p$ is Gorenstein, \cite[Corollary 5.6]{CHL2} yields that \[ \lceil\frac{2S(R_p)-f_2}{2}\rceil=\lceil \frac{f_1}{2}\rceil=\lceil \frac{a}{2}\rceil.\] 
    Therefore, we replace $\ord(c)$ with $\frac{a}{2}$ if $a$ is even, and with $b+1$ if $a=2b+1$. 
    Interchanging the order of the double sums and applying the identity $-p^{1-s}=1-(1+p^{1-s})$ yields the desired formulas.
\end{proof}
\begin{theorem}\label{thm:reformulation_for_ramlin}
    Suppose that $\sigma_p(E)=(1\ 1^2)$. Let $a=a(R_p)$ and $b=b(R_p)$ (see \eqref{eq:invariants of ramlin}).
    \begin{itemize}
        \item 
    If $a$ is even (so that $a\leq 2b$), then we have  
     \begin{align*}
        \tilde{L}_{p}(s,R)=&p^{(a+b)(1-s)}+p^{(a+b-1)(1-s)+s}\sum_{i=0}^{\frac{a}{2}-1}p^{(3s-1)i}+p^{(b+\frac{a}{2}-1)(1-s)+as}\sum_{i=0}^{b-\frac{a}{2}-1}p^{(2s-1)i}\\
        &+p^{(a+b)s}(1-p^{-s})\Bigg(\sum_{d'=0}^{b-\frac{a}{2}-1}p^{(1-2s)d'}\sum_{i=0}^{\frac{a}{2}}p^{(2-3s)i}\\
        &+p^{(b-\frac{a}{2})(1-2s)}\sum_{d'=1}^{\frac{a}{2}}p^{(3-6s)(\frac{a}{2}-d')}\Big(\sum_{i=0}^{d'}p^{(2-3s)i}+p^{1-2s}\sum_{i=0}^{d'-1}p^{(2-3s)i}+p^{2-4s}\sum_{i=0}^{d'-1}p^{(2-3s)i}\Big)\Bigg).
    \end{align*}
    \item
    If $a$ is odd (so that $a=2b+1$), then we have
      \begin{align*}
            \tilde{L}_{p}(s,R)=&p^{(3b+1)(1-s)}+p^{3b(1-s)+s}\sum_{i=0}^{b-1}p^{(3s-1)i}\\
            &+p^{(3b+1)s}(1-p^{-s})\Bigg(p^{(3-6s)b}\\
            &+\sum_{d'=1}^{b}p^{(3-6s)(b-d')}\Big(\sum_{i=0}^{d'}p^{(2-3s)i}+p^{1-2s}\sum_{i=0}^{d'}p^{(2-3s)i}+p^{2-4s}\sum_{i=0}^{d'-1}p^{(2-3s)i}\Big)\Bigg).
        \end{align*}
    \end{itemize}
\end{theorem}
\begin{proof}
    In the case that $a=0$ (thus included in the case when $a$ is even), we have 
    \begin{align*}
    \tilde{L}_p(s,R)=&p^{bs}\Big(p^{b(1-2s)}+\sum_{0<b'\leq b}p^{(b-b')(1-2s)}\Big)\\
    =& p^{b(1-s)}+p^{bs}\sum_{b'=0}^{b-1}p^{(1-2s)b'}= p^{b(1-s)}+ p^{(b-1)(1-s)}\sum_{i=0}^{b-1}p^{(2s-1)i}+(1-p^{-s})p^{bs}\sum_{d'=0}^{b-1}p^{(1-2s)d'}.
\end{align*}
This is the assertion of the theorem.
Therefore, we may and do assume that $a\geq 1$.

    \vspace{0.5em}
    \noindent \textbf{The case when $a$ is even:} \\
    We abbreviate $a=2m$ and $b=m+k$ where $a\geq 2$ and $a\leq 2b$ amount to $m\geq 1$ and $k\geq 0$, and $T=a+b=3m+k$. By Proposition \ref{prop:first formula for ramlin}, we have
    \begin{equation}\label{eq:copy of prop ramlin1}
    \tilde{L}_p(s,R) =p^{Ts}\Big(p^{T(1-2s)}+\sum_{1\leq b'\leq m+k}p^{(T-b')(1-2s)}\Big)+p^{Ts}(1-p^{-s}) P_{(1\ 1^2)}^{a,b}(s)
   \end{equation}      
    where $P_{(1\ 1^2)}^{a,b}(s)=\mathcal{A}(s)+(1+p^{1-s})\mathcal{C}(s)$ with
    \begin{align*}
         \mathcal{A}(s)=&  \sum_{\substack{0< a' \leq m   \\ 2\nmid a'}}p^{\frac{a'-1}{2}+(T-\frac{3a'-1}{2})(1-2s)}
    +\sum_{{\substack{0< a' \leq m \\  2\mid a'}}}\sum_{\substack{\frac{a'-1}{2}<b'\leq m+k-\frac{a'}{2}}}(p^{\frac{a'}{2}}-p^{\frac{a'}{2}-1})p^{(T-a'-b')(1-2s)}\\&+\sum_{\substack{m<a'\leq 2m\\ 2\mid a' }}p^{m-\frac{a'}{2}+(3m-\frac{3a'}{2})(1-2s)},\\
    \mathcal{C}(s)=&
    \sum_{\substack{0< a'\leq m\\ 2\mid a'}}\sum_{\substack{\frac{a'-1}{2}\leq b'\leq m+k-\frac{a'}{2}}}
    p^{\frac{a'}{2}-1+(T-a'-b')(1-2s)}+  \sum_{\substack{0< a'\leq m}}\sum_{\substack{m+k-\frac{a'}{2}< b'\leq m+k}}p^{m+k-b'+(T-a'-b')(1-2s)}
    \\&+\sum_{\substack{0< a'\leq m\\ 2\nmid a'}}\sum_{\substack{\frac{a'+1}{2}\leq b'\leq m+k-\frac{a'}{2}}}
    p^{\lfloor \frac{a'}{2}\rfloor +(T-a'-b')(1-2s)}+\sum_{\substack{
    m<a'\leq 2m}}\sum_{ \frac{a'}{2}+k < b'\leq m+k}p^{m+k-b'+(T-a'-b')(1-2s)}.
    \end{align*}
    We claim that this expression equals the following one, which is the assertion of the theorem after factoring out $p^{Ts}$:
    \begin{equation}\label{eq:reformulation of stat in thm of ramlin}
    \begin{aligned}
          &p^{Ts}\Big(p^{T(1-2s)}+\sum_{1\leq i\leq m+k}p^{(T-i)(1-2s)}\Big)\\
        &+p^{Ts}(1-p^{-s})\Bigg(\sum_{d'=0}^{k-1}p^{(1-2s)d'}\sum_{i=0}^m p^{(2-3s)i}
        +\sum_{d'=0}^{m-1}p^{(T-1-d')(1-2s)+d's}\sum_{i=0}^{d'-1}p^{-is}\\
        &+\sum_{d'=m}^{m+k-1}p^{(T-1-d')(1-2s)+(m-1)s}\sum_{i=0}^{m-2}p^{-is}\\
        &+p^{k(1-2s)}\sum_{d'=1}^m p^{(3-6s)(m-d')}\Big(\sum_{i=0}^{d'}p^{(2-3s)i}+p^{1-2s}\sum_{i=0}^{d'-1}p^{(2-3s)i}+p^{2-4s}\sum_{i=0}^{d'-1}p^{(2-3s)i}\Big)\Bigg).
    \end{aligned}
    \end{equation}
    Since the first line is identical with the first bracketed term in \eqref{eq:copy of prop ramlin1}, it suffices to prove that the bracket multiplied by $p^{Ts}(1-p^{-s})$ coincide.
    We decompose $\mathcal{A}(s) = \mathcal{A}_1(s) + \widetilde{\mathcal{A}}_2(s) + \mathcal{A}_3(s)$ and $\mathcal{C}(s) = \mathcal{C}_1(s) + \mathcal{C}_2(s) + \mathcal{C}_3(s) + \mathcal{C}_4(s)$ strictly matching their sequential terms as in the proof of Theorem \ref{thm:reformulation_for_unramlin}. 
    Then the negative part  $-p^{\frac{a'}{2}-1}p^{(T-a'-b')(1-2s)}$ (summed over even $a'$ and $\frac{a'-1}{2}<b'\leq m+k-\frac{a'}{2}$) in $\widetilde{\mathcal{A}}_2(s)$ is identical to $\mathcal{C}_1(s)$ and thus $\mathcal{A}_2(s)=\widetilde{\mathcal{A}}_2(s)-\mathcal{C}_1(s)$ denotes the positive part of $\widetilde{\mathcal{A}}_2(s)$.  
    On the other hand, we denote by $Q_{(1\ 1^2)}^{a,b}(s)$ the bracket multiplied by $p^{Ts}(1-p^{-s})$ of \eqref{eq:reformulation of stat in thm of ramlin},
    \begin{equation*}
\begin{aligned}
Q_{(1\ 1^2)}^{a,b}(s) = &\sum_{d'=0}^{k-1} p^{d'(1-2s)} \sum_{i=0}^{m} p^{i(2-3s)} &\text{(i)} \\
    &+ \sum_{d'=0}^{m-1} p^{(T-1-d')(1-2s) + d's} \sum_{i=0}^{d'-1} p^{-is} + \sum_{d'=m}^{m+k-1} p^{(T-1-d')(1-2s) + (m-1)s} \sum_{i=0}^{m-2} p^{-is} &\text{(ii)} \\
    &+ p^{k(1-2s)} \sum_{d'=1}^{m} p^{(3-6s)(m-d')} \Big( \sum_{i=0}^{d'} p^{i(2-3s)} + p^{1-2s}\sum_{i=0}^{d'-1} p^{i(2-3s)} + p^{2-4s}\sum_{i=0}^{d'-1} p^{i(2-3s)} \Big). &\text{(iii)}
\end{aligned}
\end{equation*}

    As in the proof of Theorem \ref{thm:reformulation_for_unramlin}, we rewrite every monomial occurring on either side in the normal form $p^{A(1-2s)+B(2-3s)}$, and record it by the pair $(A,B)$. The equality is equivalent to the equality of the two multisets of pairs obtained from $P_{(1\ 1^2)}^{a,b}(s)$ and $Q_{(1\ 1^2)}^{a,b}(s)$ since every monomial $p^{X+Ys}$ has a unique representation of the form $p^{A(1-2s)+B(2-3s)}$. 
    We compute both to show they equal the identical multiset:
    \begin{equation}\label{eq:Mset_ramlin_even}
    \begin{split}
    \mathcal{M}_{\text{even}}=
    \Big(\{(A,0): 0\leq A\leq T-1\}\sqcup \bigsqcup_{1\leq i\leq m}&\{(A,i): 0\leq A\leq T-3i\}\Big) \\ &\sqcup \Big(\bigsqcup_{1\leq i\leq m-1}\{(A,i): 2(m-i)\leq A\leq T-3i-1\}\Big).
    \end{split}
    \end{equation}
     
    \emph{$Q_{(1\ 1^2)}^{a,b}(s)$ equals $\mathcal{M}_{\text{even}}$.} 
    The first family (i) clearly corresponds to \[\bigsqcup_{0 \le i \le m} \{ (A,i) : 0 \le A \le k-1 \}.\]
    The family (ii) coincides with that in the proof of Theorem \ref{thm:reformulation_for_unramlin}. Therefore it corresponds to 
    \[\bigsqcup_{1 \le i \le m-1} \{ (A,i) : 2(m-i) \le A \le T-3i-1 \}.
    \]
    According to the proof of Theorem \ref{thm:reformulation_for_ram}, the family (iii) corresponds to
    \[
    \{(A,0): k\leq A\leq T-1\}\sqcup\Big(\bigsqcup_{1\leq i \leq m}\{(A,i): k\leq A\leq  T-3i\}
    \Big).
    \]
    These computations yield that $Q_{(1\ 1^2)}^{a,b}(s)$ corresponds to $\mathcal{M}_{\text{even}}$.

    \emph{$P_{(1\ 1^2)}^{a,b}(s)$ equals $\mathcal{M}_{\text{even}}$.} 
    By the above argument, we have
    \[
    P_{(1\ 1^2)}^{a,b}(s) = \widehat{\mathcal{A}}(s) + \widehat{\mathcal{C}}(s) + p^{1-s}\mathcal{C}(s),
    \]
    where $\widehat{\mathcal{A}}(s)=\mathcal{A}_1(s)+\mathcal{A}_2(s)+\mathcal{A}_3(s)$ and $\widehat{\mathcal{C}}(s) = \mathcal{C}_2(s) + \mathcal{C}_3(s) + \mathcal{C}_4(s)$. As in the proof of Theorem \ref{thm:reformulation_for_unramlin}, base terms $\widehat{\mathcal{A}}(s) + \widehat{\mathcal{C}}(s)$  generate pairs $(A,i)$ solely for even $i$. 
    Conversely, the shifted terms $p^{1-s}\mathcal{C}(s) = p^{(2-3s)-(1-2s)}\mathcal{C}(s)$ shift the indices to odd $i$.

    \vspace{0.5em}
    \noindent\textbf{Case 1: $i=0$ or $m$.} \\
    In the case that $i=0$, pairs $(A,0)$ arise uniquely from $\widehat{\mathcal{A}}(s)+\widehat{\mathcal{C}}(s)$. 
    From $\mathcal{A}_3(s)$ for $a'=2m$, we have $(0,0)$. 
    From $\mathcal{C}_4(s)$ and $\mathcal{C}_2(s)$ for $b'=m+k$, we have $(A,0)$ where $A$ ranges over $[1, 2m-1]$. 
    From $\mathcal{C}_3(s)$ for $a'=1$, $A$ ranges over $[2m, T-2]$.
    From $\mathcal{A}_1(s)$ for $a'=1$, we have $(0,T-1)$.
    Thus, $A$ spans the interval  $[0, T-1]$.
    In the case that $i=m$, pairs $(A,m)$ arise from $\mathcal{A}_2(s)$ where $A$ ranges over $[0,T-3m]$  if $m$ is even, and 
    arise from $\mathcal{C}_2(s)+\mathcal{C}_3(s)$ where $A$ ranges over $[0,T-3m]$
    if $m$ is odd.
    

    \vspace{0.5em}
    \noindent\textbf{Case 2: Even $i$ ($1 \le i \le m-1$).} \\
    Pairs $(A,i)$ with even $i$ arise  solely from $\widehat{\mathcal{A}}(s)+\widehat{\mathcal{C}}(s)$. 
    \begin{itemize}
        \item From $\mathcal{A}_3(s)$, we have $i=2( m-\frac{a'}{2})=2m-a'$. The $A$-coordinate is $0$.
        \item From $\mathcal{C}_4(s)$ and $\mathcal{C}_2(s)$, we have $i=2(m+k-b')$. The limits of $b'$ imply that $A$ ranges over $b'$ limits naturally bounds $A \in [1, 2m-2i-1]$.
        \item From $\mathcal{C}_3(s)$ and $\mathcal{A}_1(s)$, we have $i=a'-1$.
        The evaluations of $A$ bounds exactly $A \in [2m-2i, T-3i-1]$. 
        In particular, $(T-3i-1,i)$ is the unique pair arising from $\mathcal{A}_1(s)$.
        \item From $\mathcal{A}_2(s)$, we have $i=2\cdot \frac{a'}{2}=a'$. The $A$-coordinate ranges over $A \in [2m-2i, T-3i]$.
    \end{itemize}
    Taking the multiset union of these intervals gives $[0, T-3i] \sqcup [2(m-i), T-3i-1]$.

    \vspace{0.5em}
    \noindent\textbf{Case 3: Odd $i$ ($1 \le i \le m-1$).} \\
    Pairs $(A,i)$ with odd $i$ arise solely from $p^{1-s}\mathcal{C}(s)$.
    \begin{itemize}
        \item By $\mathbf{Case\ 1}$ and $\mathbf{Case\ 2}$, for an even $i'=i-1$ with $0\leq i' \leq m-1$, the term $\mathcal{C}_2(s)+\mathcal{C}_3(s)+\mathcal{C}_4(s)$ generates pairs $(A',i')$ where $A' \in [1, T-3i'-2]$. 
    Applying the shift $p^{1-s}=p^{(2-3s)-(1-2s)}$, 
    the term $p^{1-s}(\mathcal{C}_2(s)+\mathcal{C}_3(s)+\mathcal{C}_4(s))$ shifts these to pairs $(A,i)=(A'-1,i'+1)$, where $A$ ranges over $[0, T-3i]$.
        \item From $p^{1-s}\mathcal{C}_1(s)$, we have $i=(a'-2)+1$.
        The $A$-coordinate ranges over  $[2m-2i, T-3i-2]$.
    \end{itemize}
    The multiset union for odd $i$ is therefore $[0,T-3i]\sqcup [2(m-i),T-3i-1]$, matching the even case exactly.
    
Summing over all $0\leq i \leq m$, we conclude that the derived multiset of pairs $(A,B)$ from $P_{(1\ 1^2)}^{a,b}(s)$ is equivalent to the target multiset $\mathcal{M}_{even}$. This completes the proof for the case when $a$ is even.

    \vspace{1em}
    \noindent\textbf{The case when $a$ is odd ($a=2b+1$):} \\
    The proof is almost identical with the case when $a$ is even. 
    By Proposition \ref{prop:first formula for ramlin}, we have
    \begin{align*}
      \tilde{L}_p(s,R)=
    &p^{(3b+1)s}\Big(p^{(3b+1)(1-2s)}+\sum_{1\leq b'\leq b}p^{((3b+1)-b')(1-2s)}\Big)+p^{(3b+1)s}(1-p^{-s})P_{(1\ 1^2)}^{2b+1,b}
    \end{align*}
    where $P_{(1\ 1^2)}^{2b+1,b}=\mathcal{A}(s)+(1+p^{1-s})\mathcal{C}(s)$ with
    \begin{align*}
    \mathcal{A}(s)=&\sum_{\substack{0< a' \leq b+1  \\ 2\nmid a'}}p^{\frac{a'-1}{2}+(3b+1-\frac{3a'-1}{2})(1-2s)}+\sum_{{\substack{0< a' \leq b+1\\  2\mid a'}}}\sum_{\substack{\frac{a'}{2}\leq b'\leq b-\frac{a'}{2}}}(p^{\frac{a'}{2}}-p^{\frac{a'}{2}-1})p^{(3b+1-a'-b')(1-2s)} 
    \\&+\sum_{\substack{b+1<a'\leq 2b+1 \\ 2\nmid a'}}p^{b-\frac{a'-1}{2}+(3b+1-\frac{3a'-1}{2})(1-2s)},\\
    \mathcal{C}(s)=&\sum_{\substack{0< a'\leq b+1\\ 2\mid a'}}
    \sum_{\frac{a'}{2}\leq b'\leq b-\frac{a'}{2} }p^{\frac{a'}{2}-1+(3b+1-a'-b')(1-2s)}
    +\sum_{\substack{0< a'\leq b+1}}\sum_{b-\lfloor\frac{a'}{2}\rfloor<b'\leq b}p^{b-b'+(3b+1-a'-b')(1-2s)}\\
    &+\sum_{\substack{0< a'\leq b+1\\ 2\nmid a'}} \sum_{\substack{\frac{a'+1}{2}\leq  b'\leq b-\frac{a'-1}{2}}}p^{\frac{a'-1}{2}+(3b+1-a'-b')(1-2s)}+\sum_{b+1 <a'\leq 2b+1}\sum_{\frac{a'-1}{2} < b'\leq b}p^{b-b'+(3b+1-a'-b')(1-2s)}.
    \end{align*} 
    We claim that this coincides with the following reformulation of the formula stated in the theorem:
    \begin{align*}
            &p^{(3b+1)s}\Bigg(p^{(3b+1)(1-2s)}+\sum_{1\leq i\leq b}p^{((3b+1)-i)(1-2s)}\Bigg)\\
            &+p^{(3b+1)s}(1-p^{-s})\Bigg(p^{(3-6s)b}+\sum_{d'=0}^{b-1}p^{(3b-d')(1-2s)+d's}\sum_{i=0}^{d'-1}p^{-is}\\
            &+\sum_{d'=1}^{b}p^{(3-6s)(b-d')}\Big(\sum_{i=0}^{d'}p^{(2-3s)i}+p^{1-2s}\sum_{i=0}^{d'}p^{(2-3s)i}+p^{2-4s}\sum_{i=0}^{d'-1}p^{(2-3s)i}\Big)\Bigg).
    \end{align*}
    Here, we denote by $Q_{(1\ 1^2)}^{a,b}(s)$ the bracketed sum multiplied by $p^{(3b+1)s}(1-p^{-s})$.
    
    This is parallel to the situation of the case when $a$ is even and therefore it suffices to prove that  $P_{(1\ 1^2)}^{a,b}(s)=Q_{(1\ 1^2)}^{a,b}(s)$. 
    By identifying the form $p^{A(1-2s)+B(2-3s)}$ with the pair $(A,B)$, this is equivalent to verifying that  the two multisets of pairs obtained from $P_{(1\ 1^2)}^{a,b}(s)$ and $Q_{(1\ 1^2)}^{a,b}(s)$ coincide.
    We compute both to show that they equal the identical multiset:
    \begin{equation}
    \begin{split}
        \mathcal{M}_{\text{odd}}= \Big(\{(A,0) : 0 \le A \le 3b\} \sqcup \bigsqcup_{1 \le i \le b} &\{(A,i) : 0 \le A \le 3b-3i+1\}\Big) \\
        &\sqcup \Big(\bigsqcup_{1 \le i \le b-1} \{(A,i) : 2b-2i+1 \le A \le 3b -3i\}\Big).
    \end{split}
    \end{equation}

    \emph{$Q_{(1\ 1^2)}^{a,b}(s)$ equals $\mathcal{M}_{\text{odd}}$.} 
    The summation $\sum_{d'=0}^{b-1}p^{(3b-d')(1-2s)+d's}\sum_{i=0}^{d'-1}p^{-is}$ corresponds to the second union $\bigsqcup_{1\leq i\leq b-1}\{(A,i): A\in [2b-2i+1,3b-3i]\}$ of $\mathcal{M}_{\text{odd}}$.
    Using the argument in the proof of Theorem \ref{thm:reformulation_for_ram}, the remaining part of $Q_{(1\ 1^2)}^{a,b}(s)$ corresponds to the first union $\{(A,0) : 0 \le A \le 3b\} \sqcup \bigsqcup_{1 \le i \le b} \{(A,i) : 0 \le A \le 3b-3i+1\}$ of $\mathcal{M}_{\text{odd}}$.

    \emph{$P_{(1\ 1^2)}^{a,b}(s)$ equals $\mathcal{M}_{\text{odd}}$.} 
    Similar to the case when $a$ is even, we decompose $\mathcal{A}(s)=\mathcal{A}_1(s)+\widetilde{\mathcal{A}}_2(s)+\mathcal{A}_3(s)$ and $\mathcal{C}(s)=\mathcal{C}_1(s)+\mathcal{C}_2(s)+\mathcal{C}_3(s)+\mathcal{C}_4(s)$, where the negative part of $\widetilde{\mathcal{A}}_2(s)$ is identical to $\mathcal{C}_1(s)$ and set $\mathcal{A}_2(s)= \widetilde{\mathcal{A}}_2(s)-\mathcal{C}_1(s)$. 
    \begin{itemize}
        \item     In the case of $i=0$, pairs $(A,0)$ arise solely from $\mathcal{A}_3(s)+\mathcal{C}_4(s)+\mathcal{C}_2(s)+\mathcal{C}_3(s)+\mathcal{A}_1(s)$. The $A$-coordinate runs over $[0,3b]$.
        \item In the case that $i$ is even with $1\leq i \leq b-1$, $\mathcal{A}_3(s)+\mathcal{C}_4(s)+\mathcal{C}_2(s)+\mathcal{C}_3(s)+\mathcal{A}_1(s)$ contributes to $(A,i)$ where $A$ ranges over $[0,3b-3i]$. On the other hand, $\mathcal{A}_2(s)$ corresponds to pairs $(A,i)$ where $A\in [2b-2i+1, 3b-3i+1]$.
        \item In the case that $i$ is odd with $1\leq i \leq b-1$, $p^{1-s}(\mathcal{C}_4(s)+\mathcal{C}_2(s)+\mathcal{C}_3(s))$ contributes to $(A,i)$ where $A$ ranges over $[0,3b-3i+1]$. On the other hand, $\mathcal{C}_1(s)$ corresponds to $(A,i)$ where $A\in [2b-2i+1, 3b-3i]$.
        \item In the case of $i=b$, pairs $(A,b)$ arise from $\mathcal{A}_1(s)+\mathcal{A}_2(s)$ where $A$ ranges over $[0,1]$ if $b$ is even, and arise from $p^{1-s}(\mathcal{C}_2(s)+\mathcal{C}_3(s))$ where $A$ ranges over $[0,1]$ if $b$ is odd. 
    \end{itemize}
    Summing over all $0\leq i\leq b$, we conclude that the derived multiset of pairs $(A,B)$ from $P_{(1\ 1^2)}^{a,b}(s)$ is equivalent to the target multiset $\mathcal{M}_{\text{odd}}$.
    This completes the proof for the case when $a$ is odd.
\end{proof}

\subsubsection{The case where $\sigma_p(E)=(1\ 1\ 1)$}\label{sec:reformulation for split}
Recall the isomorphism $E_p\cong \mathbb{Q}_p\times \mathbb{Q}_p \times \mathbb{Q}_p$. 
Then, the maximal order is given by $\mathcal{O}_{E_p}= \mathbb{Z}_p\times \mathbb{Z}_p\times \mathbb{Z}_p$ under this identification.
By \cite[Proposition 5.7.(1)]{CHL2}, any order is of the form \[\mathcal{O}_{a,b,c}=\mathbb{Z}_p\langle (1,1,1),(0,p^a,0),(0,c,p^b)\rangle\] where $a,b\geq 0$ and $c\in \mathbb{Z}_p$ where 
$\begin{cases}
        c\in (p^{\lceil\frac{a}{2}\rceil}) & \text{if }a\leq 2b;\\
        c\in (p^{a-b}) \text{ or }c\in p^b+(p^{a-b}) &\text{if }a>2b.
    \end{cases}$.

\begin{lemma}\label{lem:for reformulation of split}
    Every Gorenstein order $R_p$ over $\mathbb{Z}_p$ is equivalent to $\mathbb{Z}_p\langle (1,1,1), (0,p^{a},0),(0,c,p^{b}) \rangle $ such that $\ord(c)=\ord(c-p^{b})=a/2\leq b$, up to permutation.  
\end{lemma}
\begin{proof}
    By the above argument, $R_p$ takes the form of $\mathcal{O}_{\alpha,\beta,\gamma}$ for some $\alpha,\beta\geq 0$ and $\gamma\in \mathbb{Z}_p$.
    The proof of \cite[Corollary 5.10]{CHL2} classifies the conditions for $\mathcal{O}_{\alpha,\beta,\gamma}$ to be Gorenstein as follows:
\begin{enumerate}
    \item If $\ord(\gamma) \ge \alpha$, then the Gorenstein condition is $\alpha=0$ or $\beta=0$. 
    
    \item If $\ord(\gamma) < \alpha$ and $\gamma-p^\beta \in (p^{\beta+1})$, then the Gorenstein condition is $\ord(\gamma-p^\beta) = \alpha-\beta > \beta$. 
    Here $\ord(\gamma)=\beta<\alpha$. 
    
    \item If $\ord(\gamma) < \alpha$ and $\gamma-p^\beta \notin (p^{\beta+1})$, then the Gorenstein condition is $\ord(\gamma) = \alpha/2 \le \beta$ or $\ord(\gamma) = \alpha-\beta > \beta$.
    \begin{itemize}
        \item If $\ord(\gamma) = \alpha/2 \le \beta$, then $\gamma-p^\beta \notin (p^{\beta+1})$ forces $\ord(\gamma-p^\beta) = \alpha/2$. 
        \item If $\ord(\gamma) = \alpha-\beta > \beta$, then $\ord(\gamma-p^\beta) = \beta$. 
    \end{itemize}
\end{enumerate}

In case (1), we have $\ord(\gamma) \ge \alpha$. By permuting the components of $\mathbb{Q}_p\times \mathbb{Q}_p\times \mathbb{Q}_p$ if necessary, we may and do assume that $\alpha=0$. 
Then, we have 
\[\mathbb{Z}_p\langle (1,1,1), (0,1,0),(0,\gamma,p^\beta) \rangle=\mathbb{Z}_p\langle (1,1,1), (0,1,0),(0,1,p^\beta) \rangle,\] 
since $(0,1,p^\beta) = (0,\gamma,p^\beta) + (1-\gamma)(0,1,0)$. We set $(a,b,c)=(0,\beta,1)$ if $\beta >0$. Then $a/2 = 0 \leq b$, and $\ord(c) = \ord(1) = 0 = a/2$. For $\beta>0$, one also checks $\ord(c-p^{b}) = \ord(1-p^\beta) = 0 = a/2$. 
If $\beta=0$ so that $R_p = \mathbb{Z}_p^3$, then we let $c$ be a non-trivial unit.

In case (2), we change the $\mathbb{Z}_p$-basis to obtain
\begin{align*}
\mathbb{Z}_p\langle (1,1,1), (0,p^\alpha,0),(0,\gamma,p^\beta) \rangle    
&=\mathbb{Z}_p\langle (1,1,1), (p^\alpha,0,p^\alpha),(\gamma,0,\gamma-p^\beta) \rangle\\
&=\mathbb{Z}_p\langle (1,1,1), (p^\alpha-\xi \gamma p^\beta,0,0),(\xi \gamma,0,p^{\alpha-\beta})\rangle,
\end{align*}
where
the second equality uses $\xi =p^{\alpha-\beta}/(\gamma-p^\beta)\in \mathbb{Z}_p^\times$. 
Since $\ord(\gamma p^\beta)=2\beta$ and $\alpha>2\beta$, we have $\ord(p^\alpha-\xi \gamma p^\beta)=2\beta$. 
Hence, the second generator is $(up^{2\beta}, 0, 0)$ for some unit $u \in \mathbb{Z}_p^\times$. 
Since $\ord(\gamma)=\beta$ and $\alpha>2\beta$, by scaling the second generator by a unit and permuting the first and second components, we obtain $\mathbb{Z}_p\langle (1,1,1), (0,p^{2\beta},0),(0,\xi \gamma,p^{\alpha-\beta})\rangle$. We set $(a,b,c)=(2\beta,\alpha-\beta,\xi \gamma)$.
Then, one checks that $\ord(c)=\ord(\xi \gamma)=\ord(\gamma)=\beta=a/2$. 
Since $\alpha>2\beta$, we have $a/2 < \alpha-\beta = b$ and $\ord(c-p^{b}) = \ord(\xi \gamma - p^{\alpha-\beta})=\beta=a/2$.

In case (3), if $\ord(\gamma)=\alpha/2\leq \beta$, we set $(a,b,c)=(\alpha,\beta,\gamma)$. The required conditions naturally hold. 
On the other hand, if $\ord(\gamma)=\alpha-\beta>\beta$, then we have
\begin{align*}
\mathbb{Z}_p\langle (1,1,1), (0,p^\alpha,0),(0,\gamma,p^\beta) \rangle    
&=\mathbb{Z}_p\langle (1,1,1), (p^\alpha,0,p^\alpha),(\gamma,0,\gamma-p^\beta) \rangle\\
&=\mathbb{Z}_p\langle (1,1,1), (0,0,p^\alpha-\xi p^\beta(\gamma-p^\beta)),(p^{\alpha-\beta},0,\xi(\gamma-p^{\beta}))\rangle,
\end{align*}
where $\xi=p^{\alpha-\beta}/\gamma \in \mathbb{Z}_p^\times$.
Since $\ord(\gamma-p^\beta)=\beta$ and $\alpha>2\beta$, we have $\ord(p^\alpha-\xi p^\beta(\gamma-p^\beta))=2\beta$. Hence, the second generator is $(0, 0, up^{2\beta})$ for some unit $u \in \mathbb{Z}_p^\times$.
By scaling the second generator by a unit $u^{-1}$ and applying the cycle permutation $(1\ 3 \ 2)$,
we obtain the order $\mathbb{Z}_p\langle (1,1,1), (0,p^{2\beta},0),(0,\xi(\gamma-p^{\beta}),p^{\alpha-\beta})\rangle$. 
We set $(a,b,c)=(2\beta,\alpha-\beta,\xi (\gamma-p^\beta))$.
Then, one checks that $\ord(c)=\ord(\xi(\gamma-p^\beta))=\ord(\gamma-p^\beta)=\beta=a/2$. 
Since $\alpha>2\beta$, we have $a/2 < \alpha-\beta = b$ and $\ord(c-p^{b}) = \ord(\xi(\gamma-p^\beta) - p^{\alpha-\beta})=\beta=a/2$.
\end{proof}

\begin{remark}\label{rmk:about a b in split}
    As in the cases $\sigma_p(E)=(1\ 2)$ and $\sigma_p(E)=(1\ 1^2)$, the integers $a$ and $b$ appearing in Lemma \ref{lem:for reformulation of split} are intrinsic invariants of $R_p$, which are uniquely determined regardless of the choice of generators.
    According to Lemma \ref{lem:for reformulation of split}, by permuting the components of $\mathbb{Z}_p\times \mathbb{Z}_p\times \mathbb{Z}_p$ if necessary, we may and do assume that $R_p=\mathcal{O}_{a,b,c}$ with $\ord(c)=\ord(c-p^b)=a/2\leq b$.
    Let $(p^{f_1})\times (p^{f_2})\times (p^{f_3})$ be the conductor $f(\mathcal{O}_{a,b,c})$ of $\mathcal{O}_{a,b,c}$ in $\mathcal{O}_{E_p}$.
    Then by the proof of \cite[Corollary 5.10]{CHL2}, we have $a=f_2$ and $b=f_1-\frac{f_2}{2}$.
    Henceforth, we denote such $a$ and $b$ by $a(R_p)$ and $b(R_p)$, respectively.
\end{remark}

\begin{proposition}\label{prop:first formula for split}
    Suppose that $\sigma_p(E)=(1\ 1\ 1)$. Let $a=a(R_p)$ and $b=b(R_p)$ (see Remark \ref{rmk:about a b in split}).
    Then we have
\begin{equation*}
\begin{aligned}
    \tilde{L}_p(s,R)=&p^{(a+b)s}\Bigg(p^{(a+b)(1-2s)}+(1-p^{-s})\bigg(\sum_{0<b'\leq b}p^{(a+b-b')(1-2s)}+\sum_{0<a'\leq \frac{a}{2}}2p^{(a+b-a')(1-2s)}\bigg)\Bigg)\\
         &+p^{(a+b)s}(1-p^{-s})^2 P_{(1\ 1\ 1)}^{a,b}(s)\end{aligned}
\end{equation*}
where 
         \begin{align*}
             P_{(1\ 1\ 1)}^{a,b}(s)=&\sum_{0 < a' \leq \frac{a}{2}} \sum_{0 < b' < \frac{a'}{2}} 2(p^{b'}-p^{b'-1})p^{(a+b-a'-b')(1-2s)}+\sum_{\substack{\frac{a}{2} < a' \leq a \\ 2 \mid a'}} p^{\frac{a}{2}-\frac{a'}{2}+(\frac{3a}{2}-\frac{3a'}{2})(1-2s)} \\
         &+\sum_{\substack{0 < a' \leq \frac{a}{2}\\ 2 \mid a'}} \Big((p^{\frac{a'}{2}}-2p^{\frac{a'}{2}-1})p^{(a+b-\frac{3}{2}a')(1-2s)}+\sum_{\frac{a'}{2}< b'\leq b-\frac{a'}{2}}(p^{\frac{a'}{2}}-p^{\frac{a'}{2}-1})p^{(a+b-a'-b')(1-2s)}\Big)
         \\&
         +(1+p^{1-s})\bigg(\sum_{0 < a' \leq \frac{a}{2}} \sum_{0 < b' < \frac{a'}{2}} 2p^{b'-1+(a+b-a'-b')(1-2s)}\\&+\sum_{0<a'\leq \frac{a}{2},\ 2|a'}\Big(2p^{\frac{a'}{2}-1+(a+b-\frac{3}{2}a')(1-2s)}+\sum_{\frac{a'}{2}< b'\leq b-\frac{a'}{2}}
        p^{\frac{a'}{2}-1+(a+b-a'-b')(1-2s)}
        \\&+\sum_{b-\frac{a'}{2}<b' \leq b}p^{b-b'+(a+b-a'-b')(1-2s)}\Big)+\sum_{0<a'\leq \frac{a}{2},\ 2\nmid a'}\sum_{\frac{a'}{2}\leq b'\leq b}
        p^{\min \left(\lfloor \frac{a'}{2} \rfloor, b-b' \right)+(a+b-a'-b')(1-2s)} \\
        &+\sum_{\frac{a}{2}<a'\leq a}\sum_{ \frac{a'}{2}+b-\frac{a}{2}< b' \leq b}p^{b-b'+(a+b-a'-b')(1-2s)}\bigg).
\end{align*}    
\end{proposition}
\begin{proof}
Since the structure of the ideal class monoid of an order in $\mathcal{O}_{E_p}$ is invariant under permutations of the $\mathbb{Z}_p$-factors of $\mathcal{O}_{E_p}=\mathbb{Z}_p\times\mathbb{Z}_p\times\mathbb{Z}_p$, we may and do assume that $R_p=\mathcal{O}_{a,b,c}$.
By \cite[Proposition 5.7.(2)]{CHL2}, any overorder of $R_p$ is of the form $\mathcal{O}_{a',b',c'}$.
Moreover, \cite[Proposition 5.7.(4)]{CHL2} yields that $S(R_p)=a+b$ and $[\mathcal{O}_{a',b',c'}:R_p]=p^{a+b-a'-b'}$.
    If $a'=b'=0$ (so that $\mathcal{O}_{a',b',c'}= \mathcal{O}_{E_p}$), then maximal ideals of $\mathcal{O}_{a',b',c'}$ are $(p)\times \mathbb{Z}_p\times \mathbb{Z}_p$, $\mathbb{Z}_p\times (p)\times \mathbb{Z}_p$, and $\mathbb{Z}_p\times \mathbb{Z}_p\times (p)$, and thus $\zeta_{\mathcal{O}_{a',b',c'}}(s)=\frac{1}{(1-p^{-s})^3}$.
    If exactly one of $a'$ and $b'$ is  $0$, 
    then $\mathcal{O}_{a',b',c'} \cong \mathbb{Z}_p\times R_p'$ where $R_p'\cong \mathbb{Z}_p\langle (1,1),(p^{\max(a',b')},0) \rangle\subset \mathbb{Z}_p\times \mathbb{Z}_p$ and thus $\zeta_{\mathcal{O}_{a',b',c'}}(s)=\frac{1}{(1-p^{-s})^2}$.
    Otherwise, $\mathcal{O}_{a',b',c'}$ is a local ring with the residue field $\mathbb{F}_p$.
    Therefore we have
\[\frac{\zeta_{\mathcal{O}_{a',b',c'}}(s)}{\zeta_{E_p}(s)}=
\left\{
\begin{array}{l l}
     1&  \text{if $a'=b'=0$};\\
     1-p^{-s}& \text{if exactly one of $a'$ and $b
     '$ is $0$};\\
     (1-p^{-s})^2 &\text{otherwise}.
\end{array}
\right.
\]
Then, the counting of overorders and Gorenstein overorders $\mathcal{O}_{a',b',c'}$ of $R_p=\mathcal{O}_{a,b,c}$ for fixed $a'$ and $b'$ in the proof of \cite[Theorem 5.9]{CHL2} yields that
\begin{align*}
    \tilde{L}_p(s,R)=
        &p^{(a+b)s}\Bigg(p^{(a+b)(1-2s)}+(1-p^{-s})\Big(\sum_{0<b'\leq b}p^{(a+b-b')(1-2s)}+\sum_{0<a'\leq \frac{a}{2}}2p^{(a+b-a')(1-2s)}\Big)\Bigg)\\
        &+p^{(a+b)s}(1-p^{-s})^2\Bigg((1+p^{1-s})\sum_{0< b'\leq b}\Big(
        \sum_{0< a'\leq \min(\frac{a}{2}, 2b')}p^{\min(\lfloor a'/2 \rfloor,b-b')+(a+b-a'-b')(1-2s)}
        \\&+\sum_{2b'<a'\leq \frac{a}{2}}2p^{b'+(a+b-a'-b')(1-2s)}+\sum\limits_{\substack{\frac{a}{2}<a'\leq \min(a, 2b')\\ \lceil a'/2 \rceil-b'\leq \frac{a}{2}-b}}p^{b-b'+(a+b-a'-b')(1-2s)}\Big)\\
        &\\
        &
        -p^{1-s}\Big(\sum_{\substack{0<b'\leq b}}\sum\limits_{\substack{0<a'\leq \min(\frac{a}{2}, 2b')\\ a'/2\leq b-b', ~2\mid a'}}(p^{a'/2}-p^{a'/2-1})p^{(a+b-a'-b')(1-2s)}
        -\sum\limits_{\substack{0<a'\leq\frac{a}{2}\\ 2\mid a'}}p^{a'/2-1+(a+b-\frac{3a'}{2})(1-2s)}\\
        &
        +\sum_{0<b'\leq b}\sum_{2b'<a'\leq \frac{a}{2}}2(p^{b'}-p^{b'-1})p^{(a+b-a'-b')(1-2s)}
        +\sum_{\substack{\frac{a}{2}<a'\leq a\\ ~2\mid a'}}p^{\frac{a}{2}-a'/2+(\frac{3a}{2}-\frac{3a'}{2})(1-2s)}\Big)\Bigg).
    \end{align*}
        Here, we apply the following reformulation to a component of $\tilde{L}_p(s,R)$:
    \begin{align*}
     &p^{S(R_p)s}(1-p^{-s})^2
    \sum_{\substack{R_p\subset\mathcal{O}_p\subset \mathcal{O}_{E_p}\\ \zeta_{\mathcal{O}_p}(s)/\zeta_{E_p}(s)=(1-p^{-s})^2}}h_{\mathcal{O}_p}(s)[\mathcal{O}_p:R_p]^{1-2s} 
    \\&=
         p^{S(R_p)s}(1-p^{-s})^2\Big(
    \sum_{\substack{R_p\subset\mathcal{O}_p\subset \mathcal{O}_{E_p}\\ \zeta_{\mathcal{O}_p}(s)/\zeta_{E_p}(s)=(1-p^{-s})^2}}(1+p^{1-s})[\mathcal{O}_p:R_p]^{1-2s} - \sum_{\substack{R_p\subset\mathcal{O}_p\subset \mathcal{O}_{E_p}\\ \zeta_{\mathcal{O}_p}(s)/\zeta_{E_p}(s)=(1-p^{-s})^2\\ \mathcal{O}_p :\ Gorenstein}}p^{1-s}[\mathcal{O}_p:R_p]^{1-2s}
    \Big).
    \end{align*}
    Interchanging the order of summation and applying the identity $-p^{1-s}=1-(1+p^{1-s})$ yields the desired formula.
\end{proof}
\begin{theorem}\label{thm:reformulation_for_split}
Suppose that $\sigma_p(E)=(1\ 1\ 1)$. Let $a=a(R_p)$ and $b=b(R_p)$ (see Remark \ref{rmk:about a b in split}). 
Then we have
\begin{align*}
    \tilde{L}_p(s,R) =& p^{(a+b)(1-s)}+(1-p^{-s})\left( p^{(a+b)(1-s)+2s-1}\sum_{i=0}^{\frac{a}{2}-1}3p^{(3s-1)i}+p^{(b+\frac{a}{2}-1)(1-s)+as}\sum_{i=0}^{b-\frac{a}{2}-1}p^{(2s-1)i} \right)\\
    &+(1-p^{-s})^2 p^{(a+b)s}\Bigg(\sum_{d'=0}^{b-\frac{a}{2}-1}p^{(1-2s)d'}\sum_{i=0}^{\frac{a}{2}}p^{(2-3s)i} \\
    &+p^{(b-\frac{a}{2})(1-2s)}\sum_{d'=1}^{\frac{a}{2}}p^{(3-6s)(\frac{a}{2}-d')}\left(\sum_{i=0}^{d'}p^{(2-3s)i}+p^{1-2s}\sum_{i=0}^{d'-1}p^{(2-3s)i}+p^{2-4s}\sum_{i=0}^{d'-2}p^{(2-3s)i}\right)\Bigg).
\end{align*}
\end{theorem}
\begin{proof}
By Proposition \ref{prop:first formula for split}, we have
\begin{equation}\label{eq:full_LHS of split}
\begin{aligned}
    \tilde{L}_p(s,R)=&p^{(a+b)s}\Bigg(p^{(a+b)(1-2s)}+(1-p^{-s})\bigg(\sum_{0<b'\leq b}p^{(a+b-b')(1-2s)}+\sum_{0<a'\leq \frac{a}{2}}2p^{(a+b-a')(1-2s)}\bigg)\Bigg)\\
         &+p^{(a+b)s}(1-p^{-s})^2 P^{a,b}_{(1\ 1\ 1)}(s).
\end{aligned}
\end{equation}
We note that 
\begin{align*}
    P_{(1\ 1\ 1)}^{a,b}(s)=&\sum_{0 < a' \leq \frac{a}{2}} \sum_{0 < b' < \frac{a'}{2}} 2(p^{b'}-p^{b'-1})p^{(a+b-a'-b')(1-2s)}-\sum_{\substack{0 < a' \leq \frac{a}{2}\\ 2 \mid a'}}2p^{\frac{a'}{2}-1+(a+b-\frac{3}{2}a')(1-2s)} \\
        &+(1+p^{1-s})\Big(\sum_{0 < a' \leq \frac{a}{2}} \sum_{0 < b' < \frac{a'}{2}} 2p^{b'-1+(a+b-a'-b')(1-2s)}+\sum_{0<a'\leq \frac{a}{2},\ 2|a'}2p^{\frac{a'}{2}-1+(a+b-\frac{3}{2}a')(1-2s)}\Big)\\
        &+P_{(1\ 2)}^{a,b}(s).
\end{align*}

In the case that $a=0$, we have 
\begin{align*}
       \tilde{L}_p(s,R)=&p^{bs}\Big(p^{b(1-2s)}+(1-p^{-s})\sum_{0<b'\leq b}p^{(b-b')(1-2s)}\Big)\\
    =& p^{b(1-s)}+(1-p^{-s})p^{bs}\sum_{b'=0}^{b-1}p^{(1-2s)b'}\\
    =& p^{b(1-s)}+ (1-p^{-s})\Big( p^{(b-1)(1-s)}\sum_{i=0}^{b-1}p^{(2s-1)i}+(1-p^{-s})p^{bs}\sum_{d'=0}^{b-1}p^{(1-2s)d'}\Big).
\end{align*}
It proves the theorem when $a=0$ and thus we may and do assume that $a\geq 2$ (note that $2|a$ by Remark \ref{rmk:about a b in split}).

In the case that $a\geq 2$, after factoring out $p^{(a+b)s}$, the assertion of the theorem is equivalent to
\begin{align*}
    &p^{(a+b)s}\Bigg(p^{(a+b)(1-2s)}+(1-p^{-s})\bigg(\sum_{0<b'\leq b}p^{(a+b-b')(1-2s)}+\sum_{0<a'\leq \frac{a}{2}}2p^{(a+b-a')(1-2s)}\bigg)\Bigg)\\
    &+p^{(a+b)s}(1-p^{-s})^2 Q_{(1\ 1\ 1)}^{a,b}(s).    
    \end{align*}
where the first line is identical to that of (\ref{eq:full_LHS of split}) and
    \begin{align*}
     Q_{(1\ 1\ 1)}^{a,b}(s)=&\sum_{d'=0}^{b-\frac{a}{2}-1}p^{(1-2s)d'}\sum_{i=0}^{\frac{a}{2}}p^{(2-3s)i}  +3\sum_{d'=0}^{a/2-1}p^{(a+b-1-d')(1-2s)+d's}\sum_{i=0}^{d'-1}p^{-is}\\
        &+\sum_{d'=a/2}^{b-1}p^{(a+b-1-d')(1-2s)+(a/2-1)s}\sum_{i=0}^{a/2-2}p^{-is}\\
    &+p^{(b-\frac{a}{2})(1-2s)}\sum_{d'=1}^{\frac{a}{2}}p^{(3-6s)(\frac{a}{2}-d')}\left(\sum_{i=0}^{d'}p^{(2-3s)i}+p^{1-2s}\sum_{i=0}^{d'-1}p^{(2-3s)i}+p^{2-4s}\sum_{i=0}^{d'-2}p^{(2-3s)i}\right).
\end{align*}
We note that
 \begin{align*}
     Q_{(1\ 1\ 1)}^{a,b}(s)=&2\sum_{d'=0}^{a/2-1}p^{(a+b-1-d')(1-2s)+d's}\sum_{i=0}^{d'-1}p^{-is}+Q_{(1\ 2)}^{a,b}(s)   
\end{align*}
where $Q_{(1\ 2)}^{a,b}(s)$ is defined in \eqref{eq:rhs_of_unramsplit}. 

Since $a$ is even and satisfies $a\leq 2b$ by Lemma \ref{lem:for reformulation of split}, the proof of Theorem  \ref{thm:reformulation_for_unramlin} (see (\ref{eq:comparison_splitunram})) yields $P_{(1\ 2)}^{a,b}(s)=Q_{(1\ 2)}^{a,b}(s)$ and hence it suffices to show that
\[
P^{a,b}_{(1\ 1\ 1)}(s)-P^{a,b}_{(1\ 2)}(s)=Q^{a,b}_{(1\ 1\ 1)}(s)-Q^{a,b}_{(1\ 2)}(s).
\]
This assertion is equivalent to verifying that
\begin{equation}\label{eq:final equation of reformulation split case}
\begin{aligned}
    &\sum_{0 < a' \leq \frac{a}{2}} \sum_{0 < b' < \frac{a'}{2}} 2(p^{b'}-p^{b'-1})p^{(a+b-a'-b')(1-2s)}-\sum_{\substack{0 < a' \leq \frac{a}{2}\\ 2 \mid a'}}2p^{\frac{a'}{2}-1+(a+b-\frac{3}{2}a')(1-2s)} \\
        &+(1+p^{1-s})\Big(\sum_{0 < a' \leq \frac{a}{2}} \sum_{0 < b' < \frac{a'}{2}} 2p^{b'-1+(a+b-a'-b')(1-2s)}+\sum_{\substack{0<a'\leq \frac{a}{2}\\ 2 \mid a'}}2p^{\frac{a'}{2}-1+(a+b-\frac{3}{2}a')(1-2s)}\Big)\\
        &=2\sum_{d'=0}^{a/2-1}p^{(a+b-1-d')(1-2s)+d's}\sum_{i=0}^{d'-1}p^{-is}.
\end{aligned}
\end{equation}
On the left-hand side of (\ref{eq:final equation of reformulation split case}), distributing the factor $1+p^{1-s}$ and simplifying the expression yields
\begin{equation}\label{eq:equation2 in split}
\begin{aligned}
    &\sum_{0 < a' \leq \frac{a}{2}} \sum_{0 < b' < \frac{a'}{2}} 2p^{b'+(a+b-a'-b')(1-2s)}+\sum_{0 < a' \leq \frac{a}{2}} \sum_{0 < b' < \frac{a'}{2}} 2p^{b'+(a+b-a'-b')(1-2s)-s}\\
        &+\sum_{\substack{0<a'\leq \frac{a}{2}\\ 2 \mid a'}}2p^{\frac{a'}{2}+(a+b-\frac{3}{2}a')(1-2s)-s}.
\end{aligned}
\end{equation}
For a fixed $a'$, the set of exponents of $p$ in the first summation on the first line is $\{(a+b-a')(1-2s)+2b's\}_{1\leq b' <\frac{a'}{2}}$, and that in the second summation is $\{(a+b-a')(1-2s)+(2b'-1)s\}_{1\leq b' <\frac{a'}{2}}$. 
Taking their union, we obtain $\{(a+b-a')(1-2s)+is\}_{1\leq i \leq a'-2}$ if $a'$ is even, and $\{(a+b-a')(1-2s)+is\}_{1\leq i \leq a'-1}$ if $a'$ is odd.
Then, by relabeling the index $a'$ as $d'$, the first line of (\ref{eq:equation2 in split}) can be rewritten as follows:
\begin{equation}\label{eq:equation3 in split}
\sum_{\substack{1 \leq d' \leq \frac{a}{2}\\ 2 \mid d'}} \sum_{1 \leq  i \leq d'-2} 2p^{(a+b-d')(1-2s)+is}+\sum_{\substack{1 \leq d' \leq \frac{a}{2}\\ 2 \nmid d'}} \sum_{1 \leq  i \leq d'-1} 2p^{(a+b-d')(1-2s)+is}.
\end{equation}
On the other hand, the last term in (\ref{eq:equation2 in split}) is equal to
\[
\sum_{\substack{0<d'\leq \frac{a}{2}\\ 2 \mid d'}}2p^{(a+b-d')(1-2s)+(d'-1)s},
\]
which precisely provides the missing terms for $i=d'-1$ in the first summation of (\ref{eq:equation3 in split}). Combining these, we obtain
\[
\sum_{1\leq d'\leq \frac{a}{2}}\sum_{1\leq i\leq d'-1}2p^{(a+b-d')(1-2s)+is}=\sum_{d'=0}^{\frac{a}{2}-1}\sum_{ i=1}^{d'} 2p^{(a+b-1-d')(1-2s)+is}=\sum_{d'=0}^{\frac{a}{2}-1}\sum_{ i=0}^{d'-1} 2p^{(a+b-1-d')(1-2s)+(d'-i)s}.
\]
This completes the proof.
\end{proof}

\section{Functional equation of the $L$-function for a cubic order}
We now prove the main theorem.
\begin{theorem}\label{thm:main thm}
    For a Gorenstein order $R$ of a cubic number field $E$, we have
    \[
    \Lambda(s,R)=\Lambda(1-s,R).
    \]
\end{theorem}
\begin{proof}
    By Lemma \ref{lem:definition of tilde L}, it suffices to show that $\tilde{L}(s,R)=\tilde{L}(1-s,R)$.
    By Proposition \ref{prop:euler_product} and Remark \ref{rmk:finite alive for local factors}, this reduces to proving the local functional equations,  $\tilde{L}_p(s,R)=\tilde{L}_p(1-s,R)$ which are proved in Theorems \ref{prop:local_ftnl_eqn}-\ref{thm:functional_equation_quasisplit}. 
\end{proof}
It therefore remains to prove the local functional equations.
Throughout this section, we derive them using the reformulations obtained in Section \ref{sec:reformulation of local factors}.
\subsection{Local functional equation at $p$ where $E$ is irreducible}
Before proceeding to the main theorem, we first represent a key lemma which will be used to verify the local functional equations for a prime $p$ where $E$ is irreducible.
\begin{lemma}\label{lem:double sum formula for nonsplit case}
For any integer $d\geq 0$, we have the identity
\begin{equation}\label{eq:lemma for irred case}
p^{3d(1-s)}\sum_{d'=0}^d p^{3d'(2s-1)}\sum_{i=0}^{d'}p^{(2-3s)i}=p^{3ds}\sum_{d'=0}^d p^{3d'(1-2s)}\sum_{i=0}^{d'}p^{(3s-1)i}.
\end{equation}
Consequently, we also obtain the following two identities:
\begin{gather*}
    p^{3(d+1)(1-s)}\sum_{d'=1}^{d} p^{3d'(2s-1)}\sum_{i=0}^{d'-1}p^{(2-3s)i} = p^{3(d+1)s}\sum_{d'=1}^{d} p^{3d'(1-2s)}\sum_{i=0}^{d'-1}p^{(3s-1)i}, \textit{ and} \\
    p^{3(d+2)(1-s)}\sum_{d'=2}^{d} p^{3d'(2s-1)}\sum_{i=0}^{d'-2}p^{(2-3s)i} = p^{3(d+2)s}\sum_{d'=2}^{d} p^{3d'(1-2s)}\sum_{i=0}^{d'-2}p^{(3s-1)i}.
\end{gather*}
\end{lemma}

\begin{proof}
Replacing $d'$ with $d-d'$ and subsequently substituting the inner summation index $i$ with $i-d'$, the left-hand side of \eqref{eq:lemma for irred case} can be expanded as
\begin{align*}
p^{3d(1-s)}\sum_{d'=0}^d\sum_{i=0}^{d'}p^{3d'(2s-1)+(2-3s)i}
&=p^{3ds}\sum_{d'=0}^d\sum_{i=0}^{d-d'} p^{3d'(1-2s)+(2-3s)i}\\
&=p^{3ds}\sum_{d'=0}^d\sum_{i=d'}^{d} p^{(1-3s)d'+(2-3s)i}.
\end{align*}
Next, interchanging the dummy variables $i$ and $d'$ yields
\[
p^{3ds}\sum_{i=0}^d\sum_{d'=i}^{d} p^{(1-3s)i+(2-3s)d'}.
\]
Since the summation domain is $0 \leq i \leq d' \leq d$, changing the order of summation gives
\[
p^{3ds}\sum_{d'=0}^d\sum_{i=0}^{d'} p^{(1-3s)i+(2-3s)d'}.
\]
Finally, substituting $i \mapsto d'-i$ in the inner sum, we obtain
\[
p^{3ds}\sum_{d'=0}^d p^{3d'(1-2s)}\sum_{i=0}^{d'} p^{(3s-1)i},
\]
which yields \eqref{eq:lemma for irred case}.

The remaining two identities follow directly from \eqref{eq:lemma for irred case}. Applying the main identity with $d$ replaced by $d-1$, shifting the outer summation index via $d' \mapsto d'-1$, and multiplying both sides by $p^3$ yields the first identity. Similarly, replacing $d$ with $d-2$, shifting the outer index via $d' \mapsto d'-2$, and multiplying both sides by $p^6$ gives the second identity.
\end{proof}

\begin{theorem}\label{prop:local_ftnl_eqn}
Suppose that $E$ is irreducible at $p$. Then we have the local functional equation:
\[
    \tilde{L}_p(s,R)=\tilde{L}_p(1-s,R).
\]
\end{theorem}
\begin{proof}
    \begin{enumerate}
        \item 
        By Theorem \ref{thm:reformulation_for_unram}, for a prime $p$ such that $\sigma_p(E)=(3)$ and $d=S(R_p)/3$, we have
\begin{align*}
    \tilde{L}_p(s,R)=&
    p^{3d(1-s)}+(1+p^{-s}+p^{-2s})\sum_{d'=1}^{d}p^{3d(1-s)-3d'(1-2s)}\left(\sum_{i=0}^{d'}p^{(2-3s)i}+p^{1-2s}\sum_{i=0}^{d'-1}p^{(2-3s)i}+p^{2-4s}\sum_{i=0}^{d'-2}p^{(2-3s)i}\right)\\
    =&p^{3d(1-s)}+p^{3d(1-s)}\sum_{d'=1}^{d}p^{3d'(2s-1)}\left(\sum_{i=0}^{d'}p^{(2-3s)i}+p^{1-2s}\sum_{i=0}^{d'-1}p^{(2-3s)i}+p^{2-4s}\sum_{i=-1}^{d'-2}p^{(2-3s)i}\right)\\
    &+p^{3d(1-s)}\sum_{d'=1}^{d}p^{3d'(2s-1)-s}\left(\sum_{i=1}^{d'}p^{(2-3s)i}+p^{1-2s}\sum_{i=0}^{d'-1}p^{(2-3s)i}+p^{2-4s}\sum_{i=-1}^{d'-2}p^{(2-3s)i}\right)\\
    &+p^{3d(1-s)}\sum_{d'=1}^{d}p^{3d'(2s-1)-2s}\left(\sum_{i=1}^{d'}p^{(2-3s)i}+p^{1-2s}\sum_{i=0}^{d'-1}p^{(2-3s)i}+p^{2-4s}\sum_{i=0}^{d'-2}p^{(2-3s)i}\right).
\end{align*}
Here, the shifted indices in the second equality arise from extracting the $(i=0)$-terms from the sums $\sum_{i=0}^{d'}p^{(2-3s)i}$ in the second and third groups. These extracted terms can be absorbed into the adjacent sums with coefficient $p^{2-4s}$, extending their lower limits from $0$ to ${-1}$.
Then, by reindexing $i$ so that all inner sums begin at $i=0$, we group the terms into three groups based on their upper limits ($d'$, $d'-1$, and $d'-2$):
\begin{align*}
    &p^{3d(1-s)}\sum_{d'=0}^dp^{3d'(2s-1)}\sum_{i=0}^{d'}p^{(2-3s)i}+p^{3(d+2)(1-s)-4}\sum_{d'=2}^{d}p^{3d'(2s-1)}\sum_{i=0}^{d'-2}p^{(2-3s)i}\\
    &+p^{3(d+1)(1-s)}\sum_{d'=1}^{d}p^{3d'(2s-1)}\sum_{i=0}^{d'-1}p^{(2-3s)i}(p^{-2+s}+p^{-1-s}+p^{-3+2s}+p^{-2}+p^{-1-2s}+p^{-3+s}+p^{-2-s}).
\end{align*}
Applying Lemma \ref{lem:double sum formula for nonsplit case} to the above expression yields
\begin{equation}\label{eq:equation for ftn eq for unram}
\begin{aligned}
    &p^{3ds}\sum_{d'=0}^dp^{3d'(1-2s)}\sum_{i=0}^{d'}p^{(3s-1)i}+p^{3(d+2)s-4}\sum_{d'=2}^{d}p^{3d'(1-2s)}\sum_{i=0}^{d'-2}p^{(3s-1)i}\\
    &+p^{3(d+1)s}\sum_{d'=1}^{d}p^{3d'(1-2s)}\sum_{i=0}^{d'-1}p^{(3s-1)i}(p^{-2+s}+p^{-1-s}+p^{-3+2s}+p^{-2}+p^{-1-2s}+p^{-3+s}+p^{-2-s}).
\end{aligned}
\end{equation}
We modify the last term of \eqref{eq:equation for ftn eq for unram} by distributing $p^{3s}$ from the factor $p^{3(d+1)s} = p^{3ds} \cdot p^{3s}$ into the polynomial, which splits the sum into three groups based on powers of $p^{s-1}$:
\begin{align*}
    &p^{3ds}\sum_{d'=1}^{d}p^{3d'(1-2s)}\sum_{i=0}^{d'-1}p^{(3s-1)i}(p^{4s-2}+p^{2s-1}+p^{5s-3}+p^{3s-2}+p^{s-1}+p^{4s-3}+p^{2s-2})\\
    &=p^{3ds}\sum_{d'=1}^{d}p^{3d'(1-2s)}\left(p^{4s-2}\sum_{i=0}^{d'-1}p^{(3s-1)i}+p^{2s-1}\sum_{i=0}^{d'-1}p^{(3s-1)i}\right)\\
    &\phantom{=}+p^{3ds}\sum_{d'=1}^{d}p^{3d'(1-2s)+s-1}\left(p^{4s-2}\sum_{i=0}^{d'-1}p^{(3s-1)i}+p^{2s-1}\sum_{i=0}^{d'-1}p^{(3s-1)i}+\sum_{i=0}^{d'-1}p^{(3s-1)i}\right)\\
    &\phantom{=}+p^{3ds}\sum_{d'=1}^{d}p^{3d'(1-2s)+2s-2}\left(p^{2s-1}\sum_{i=0}^{d'-1}p^{(3s-1)i}+\sum_{i=0}^{d'-1}p^{(3s-1)i}\right).
\end{align*}
Using the identity $p^{4s-2}\cdot p^{(3s-1)(d'-1)}=p^{(3s-1)d'+s-1}$, we extract the highest-index terms (i.e., $i=d'-1$) from the sums with the coefficient $p^{4s-2}$ in the first and second groups. These extracted terms perfectly absorb into the adjacent groups, extending their upper limits from $d'-1$ to $d'$. This transforms the expression into:
\begin{align*}
    &p^{3ds}\sum_{d'=1}^{d}p^{3d'(1-2s)}\sum_{i=0}^{d'-1}p^{(3s-1)i}(p^{4s-2}+p^{2s-1}+p^{5s-3}+p^{3s-2}+p^{s-1}+p^{4s-3}+p^{2s-2})\\
    &= p^{3ds}\sum_{d'=1}^{d}p^{3d'(1-2s)}\left(p^{2s-1}\sum_{i=0}^{d'-1}p^{(3s-1)i}+p^{4s-2}\sum_{i=0}^{d'-2}p^{(3s-1)i}\right)\\
    &\quad+p^{3ds}\sum_{d'=1}^{d}p^{3d'(1-2s)+s-1}\left(\sum_{i=0}^{d'}p^{(3s-1)i}+p^{2s-1}\sum_{i=0}^{d'-1}p^{(3s-1)i}+p^{4s-2}\sum_{i=0}^{d'-2}p^{(3s-1)i}\right)\\
    &\quad+p^{3ds}\sum_{d'=1}^{d}p^{3d'(1-2s)+2s-2}\left(\sum_{i=0}^{d'}p^{(3s-1)i}+p^{2s-1}\sum_{i=0}^{d'-1}p^{(3s-1)i}\right).
\end{align*}
Finally, adding the remaining terms from the first line of \eqref{eq:equation for ftn eq for unram} yields the missing terms, completing the factorization. 
Thus, we conclude that
\begin{align*}
    \tilde{L}_p(s,R)=& p^{3ds}+(1+p^{s-1}+p^{2s-2})\sum_{d'=1}^{d}p^{3ds-3d'(2s-1)}\left(\sum_{i=0}^{d'}p^{(3s-1)i}+p^{2s-1}\sum_{i=0}^{d'-1}p^{(3s-1)i}+p^{4s-2}\sum_{i=0}^{d'-2}p^{(3s-1)i}\right)\\
    =&\tilde{L}_p(1-s,R).
\end{align*}
        \item For a prime $p$ such that $\sigma_p(E)=(1^3)$, we have $3|S(R_p)$ or $3|(S(R_p)-1)$ by Section \ref{sec:reformulation for ram}.
        \begin{itemize}
            \item In the case that $3|S(R_p)$, let $d=S(R_p)/3$. 
            Then, by Theorem \ref{thm:reformulation_for_ram}, we have
           \begin{align*}
        \tilde{L}_p(s,R)=&p^{3d(1-s)}+\sum_{d'=1}^d p^{3d(1-s)-3d'(1-2s)}\Bigg(
        \sum_{i=0}^{d'}p^{(2-3s)i} + p^{1-2s}\sum_{i=0}^{d'-1}p^{(2-3s)i}+ p^{2-4s}\sum_{i=0}^{d'-1}p^{(2-3s)i}
        \Bigg)\\
        =& p^{3d(1-s)}\sum_{d'=0}^d p^{3d'(2s-1)}\sum_{i=0}^{d'}p^{(2-3s)i}+(p^{-2+s}+p^{-1-s})\Bigg(p^{3(d+1)(1-s)}\sum_{d'=1}^d p^{3d'(2s-1)}\sum_{i=0}^{d'-1}p^{(2-3s)i}\Bigg).
        \end{align*}
        By Lemma \ref{lem:double sum formula for nonsplit case}, this is equivalent to
        \begin{align*}
        &
        p^{3ds}\sum_{d'=0}^d p^{3d'(1-2s)}\sum_{i=0}^{d'}p^{(3s-1)i}+(p^{-1-s}+p^{-2+s})\Bigg(p^{3(d+1)s}\sum_{d'=1}^d p^{3d'(1-2s)}\sum_{i=0}^{d'-1}p^{(3s-1)i}\Bigg)\\
        &=p^{3ds}\sum_{d'=0}^d p^{3d'(1-2s)}\sum_{i=0}^{d'}p^{(3s-1)i}+(p^{2s-1}+p^{4s-2})\Bigg(p^{3ds}\sum_{d'=1}^d p^{3d'(1-2s)}\sum_{i=0}^{d'-1}p^{(3s-1)i}\Bigg)
        \\
        &=p^{3ds}+\sum_{d'=1}^d p^{3ds-3d'(2s-1)}\Bigg(\sum_{i=0}^{d'}p^{(3s-1)i}+p^{2s-1}\sum_{i=0}^{d'-1}p^{(3s-1)i}+p^{4s-2}\sum_{i=0}^{d'-1}p^{(3s-1)i}\Bigg).
        \end{align*}
        This yields the functional equation for $\tilde{L}_p(s,R)$.
        \item In the case that  $3|(S(R)-1)$, let $d=(S(R_p)-1)/3$. Then, by Theorem \ref{thm:reformulation_for_ram}, we have
        \begin{align*}
            \tilde{L}_p(s,R)=&\sum_{d'=0}^d p^{3d(1-s)-3d'(1-2s)+s}(\sum_{i=0}^{d'}p^{(2-3s)i}+p^{1-2s}\sum_{i=0}^{d'}p^{(2-3s)i}+p^{2-4s}\sum_{i=0}^{d'-1}p^{(2-3s)i})\\
            =&(p^s+p^{1-s})p^{3d(1-s)}\sum_{d'=0}^d p^{3d'(2s-1)}\sum_{i=0}^{d'}p^{(2-3s)i}+p^{3(d+1)(1-s)-1}
            \sum_{d'=1}^d p^{3d'(2s-1)}\sum_{i=0}^{d'-1}p^{(2-3s)i}.
        \end{align*}
        By Lemma \ref{lem:double sum formula for nonsplit case}, this is equivalent to
        \begin{align*}
            &(p^{1-s}+p^{s})p^{3ds}\sum_{d'=0}^d p^{3d'(1-2s)}\sum_{i=0}^{d'}p^{(3s-1)i}+p^{3(d+1)s-1}
            \sum_{d'=1}^d p^{3d'(1-2s)}\sum_{i=0}^{d'-1}p^{(3s-1)i}\\
            &=\sum_{d'=0}^d p^{3ds-3d'(2s-1)+1-s}(\sum_{i=0}^{d'}p^{(3s-1)i}+p^{2s-1}\sum_{i=0}^{d'}p^{(3s-1)i}+p^{4s-2}\sum_{i=0}^{d'-1}p^{(3s-1)i}).
        \end{align*}
        This yields the functional equation for  $\tilde{L}_p(s,R)$.\qedhere
        \end{itemize}
    \end{enumerate}
\end{proof}
\subsection{Local functional equation at $p$ where $E$ splits}
\begin{theorem}\label{thm:functional_equation_quasisplit}
Suppose that $E$ splits at $p$. 
\[
\tilde{L}_p(s,R)=\tilde{L}_p(1-s,R).
\]
\end{theorem}
\begin{proof}
Our strategy is to prove the functional equation simultaneously for the cases where $\sigma_p(E)=(1\ 2)$, $(1\ 1\ 1)$, and the case where $\sigma_p(E)=(1\ 1^2)$ with even $a(R_p)$, using algebraic reduction. 
The case where $\sigma_p(E)=(1\ 1^2)$ with odd $a(R_p)$ requires a separate but analogous argument.

\vspace{0.5em}
\noindent \textbf{Part 1: The Central Identity.} \\
    Let $a$ and $b$ be non-negative integers such that $a$ is even with $a \le 2b$. 
    We define $A = a/2$, $B = b - a/2$, and $T = a+b = 3A+B$. 
    Let $L_{(1\ 1^2)}^{a,b}(s)$ be the polynomial that yields the local factor of the case where $\sigma_p(E)=(1\ 1^2)$ upon substituting $a=a(R_p)$ and $b=b(R_p)$:
    \begin{align*}
    L_{(1\ 1^2)}^{a,b}(s) =& \; p^{T(1-s)} + p^{(T-1)(1-s)+s}\sum_{i=0}^{A-1}p^{(3s-1)i} + p^{(2A+B-1)(1-s)+2As}\sum_{i=0}^{B-1}p^{(2s-1)i} \\
    &+ (1-p^{-s})p^{Ts} \Bigg( \sum_{d'=0}^{B-1}p^{(1-2s)d'}\sum_{i=0}^{A}p^{(2-3s)i} \\
    &+ p^{B(1-2s)}\sum_{d'=1}^{A}p^{(3-6s)(A-d')} \Big( \sum_{i=0}^{d'}p^{(2-3s)i} + p^{1-2s}\sum_{i=0}^{d'-1}p^{(2-3s)i} + p^{2-4s}\sum_{i=0}^{d'-1}p^{(2-3s)i} \Big) \Bigg).
\end{align*}
We define the auxiliary polynomial $\mathcal{C}(s)$ as follows
\[ \mathcal{C}(s) = p^{T(1-s)-s} \sum_{d'=1}^A p^{d'(3s-1)} = p^{(T-1)(1-s)+s} \sum_{i=0}^{A-1} p^{i(3s-1)}. \]
Observe that under the substitution $s \leftrightarrow 1-s$, it transforms as
\[ \mathcal{C}(1-s) = p^{Ts+s-1} \sum_{d'=1}^A p^{d'(2-3s)}. \]

Applying telescoping cancellation to the double sums, we have
\begin{align*}
&(1-p^{2s-1})\sum_{d'=1}^Ap^{3d'(2s-1)}\Big(\sum_{i=0}^{d'}p^{(2-3s)i}+p^{1-2s}\sum_{i=0}^{d'-1}p^{(2-3s)i}+p^{2-4s}\sum_{i=0}^{d'-1}p^{(2-3s)i}\Big)\\
&=\sum_{d'=1}^{A}p^{d'(3s-1)}+p^{2s-1}-p^{(3A+1)(2s-1)}\sum_{d'=0}^{A}p^{d'(2-3s)}.
\end{align*}
Multiplying $L_{(1\ 1^2)}^{a,b}(s)$ by $(1-p^{2s-1})$ and evaluating the geometric series $(1-p^{2s-1})\sum_{i=0}^{B-1}p^{(2s-1)i} = 1-p^{B(2s-1)}$, we can substitute the telescoping identity above to simplify the expression.
Distributing and regrouping the terms, we deduce
\begin{align*}
    (1-p^{2s-1})L_{(1\ 1^2)}^{a,b}
    =&p^{T(1-s)}(1-p^{2s-1})+p^{(T-1)(1-s)+s}(1-p^{2s-1})\sum_{d'=0}^{A-1}p^{d'(3s-1)}\\
    &+ p^{(2A+B-1)(1-s)+2As}(1-p^{B(2s-1)})
        -(1-p^{-s})p^{Ts+2s-1}(1-p^{B(1-2s)})\sum_{i=0}^A p^{(2-3s)i}\\
        &+(1-p^{-s})p^{T(1-s)}
        \Big(\sum_{d'=1}^{A}p^{d'(3s-1)}+p^{2s-1}-p^{(3A+1)(2s-1)}\sum_{d'=0}^{A}p^{d'(2-3s)}\Big)\\
    =&p^{T(1-s)}(1-p^{2s-1})+p^{2A+(B-1)(1-s)}(1-p^{B(2s-1)})+(1-p^{-s})p^{(3A+B)(1-s)+2s-1}\\
    &-(1-p^{-s})\Big(p^{Ts}(p^{2s-1}-p^{(B-1)(1-2s)})+p^{(3A+1)s+(B-1)(1-s)}\Big)\sum_{d'=0}^{A}p^{d'(2-3s)}\\
    &+\Big( p^{T(1-s)}(p^{-s}-p^{s-1})
    +(1-p^{-s})p^{T(1-s)}\Big)\sum_{d'=1}^{A}p^{d'(3s-1)}\\
    =& p^{T(1-s)}+p^{2A+(B-1)(1-s)}-p^{2A+Bs+s-1}-p^{T(1-s)+s-1}\\
    &-(1-p^{-s})p^{Ts+2s-1}\sum_{d'=0}^{A}p^{d'(2-3s)}+ p^{T(1-s)}(1-p^{s-1})\sum_{d'=1}^{A}p^{d'(3s-1)}.
\end{align*}
To express the remaining summations in terms of $\mathcal{C}(s)$ and $\mathcal{C}(1-s)$, we algebraically extract the required coefficients.
By using the relation $T=3A+B$, we have
\begin{align*}
    &(1-p^{-s})p^{Ts+2s-1}\sum_{d'=0}^A p^{d'(2-3s)}=p^{(3A+B)s+2s-1}\Big((1-p^{2s-2})+p^{-s}(p^{3s-2} -1)\Big)\sum_{d'=0}^A p^{d'(2-3s)}\\
    &=p^{(3A+B)s+2s-1}(1-p^{2s-2})\sum_{d'=1}^A p^{d'(2-3s)}+p^{(3A+B)s+2s-1}(1-p^{2s-2})+p^{(3A+B)s+4s-3}(1-p^{(A+1)(2-3s)})\\
    &=p^{Ts+2s-1}(1-p^{2s-2})\sum_{d'=1}^A p^{d'(2-3s)}+p^{Ts+2s-1}-p^{2A+Bs+s-1}.
\end{align*}
Similarly, we have
\begin{align*}
    &p^{T(1-s)}(1-p^{s-1})\sum_{d'=1}^A p^{d'(3s-1)}=p^{(3A+B)(1-s)}\Big((1-p^{-2s}) +p^{-2s}(1-p^{3s-1})\Big)\sum_{d'=1}^A p^{d'(3s-1)}\\
    &=p^{T(1-s)}(1-p^{-2s})\sum_{d'=1}^A p^{d'(3s-1)}+p^{T(1-s)+s-1}(1-p^{A(3s-1)}).
\end{align*}
Substituting these relations back into the expanded expression, the residual cross-terms cancel out entirely. 
Consequently, we have the following \textbf{Key Equation}:
\begin{equation}\label{eq:Key_equation}
(1-p^{2s-1}) L_{(1\ 1^2)}^{a,b}(s) = p^{T(1-s)} - p^{Ts+2s-1} + (1-p^{-2s})p^s\mathcal{C}(s) - (1-p^{2s-2})p^s\mathcal{C}(1-s).
\end{equation}
We denote the right-hand side by $W(s)$. 
Here, $W(s)$ satisfies the following relation under the exchange $s\leftrightarrow 1-s$:
\begin{align*}
W(1-s) &= p^{Ts} - p^{T(1-s)+1-2s} + (1-p^{2s-2}) p^{1-s} \mathcal{C}(1-s) - (1-p^{-2s}) p^{1-s} \mathcal{C}(s) \\
&= -p^{1-2s} W(s).    
\end{align*}
Dividing $W(1-s)$ by $(1-p^{1-2s})$ yields that 
\[
L^{a,b}_{(1\ 1^2)}(1-s) = \frac{W(1-s)}{1-p^{1-2s}} = \frac{-p^{1-2s} W(s)}{-p^{1-2s}(1-p^{2s-1})} = \frac{W(s)}{1-p^{2s-1}} = L_{(1\ 1^2)}^{a,b}(s).
\]


Note that $L_{(1\ 1^2)}^{a,b}(s)$ and the auxiliary polynomial $\mathcal{C}(s)$ are polynomials in $p^{-s}$ and thus entire functions on $\mathbb{C}$. 
Consequently, any apparent poles introduced by division in the intermediate steps are strictly removable singularities.
By the principle of analytic continuation, verifying this formal algebraic identity guarantees the exact equality for all $s\in\mathbb{C}$.


Therefore, we have the functional equation for $\sigma_p(E)=(1\ 1^2)$ when $a(R_p)$ is even (so that $a(R_p)\leq 2b(R_p)$ by Section \ref{sec:reformulation for ramlin}): \[\tilde{L}_p(s,R) = L^{a(R_p),b(R_p)}_{(1\ 1^2)}(s) = L^{a(R_p),b(R_p)}_{(1\ 1^2)}(1-s)=\tilde{L}_p(1-s,R).\]

\vspace{0.5em}
\noindent \textbf{Part 2: Algebraic Reductions for the cases $\sigma_p(E)=(1\ 2)$ and $\sigma_p(E)=(1\ 1\ 1)$.} \\
Let $a=a(R_p)$ and $b=b(R_p)$. Note that $a$ is even and $a\leq 2b$ by Section \ref{sec:reformulation for unramlin} and Remark \ref{rmk:about a b in split}.
By comparing the explicit evaluations in the previous theorems, we express algebraically the local factors at $p$ with $\sigma_p(E)=(1\ 2)$ or $\sigma_p(E)=(1\ 1\ 1)$ in terms of $L_{(1\ 1^2)}^{a,b}(s)$ and $\mathcal{C}(s)$.
 \begin{itemize}
     \item  In the case where $\sigma_p(E)=(1\ 2)$, by Theorem \ref{thm:reformulation_for_unramlin}, we have
\begin{align*}
\tilde{L}_p(s,R)=&p^{T(1-s)}+(1+p^{-s})\Big(L_{(1\ 1^2)}^{a,b}(s) -p^{T(1-s)}-(1-p^{-s})\mathcal{C}(s)\Big)\\
=&(1+p^{-s})L_{(1\ 1^2)}^{a,b}(s) -p^{T(1-s)-s}-(1-p^{-2s})\mathcal{C}(s).
\end{align*}
To verify its functional equation, we evaluate $\tilde{L}_p(s,R) - \tilde{L}_p(1-s,R) = 0$.
Using the symmetry $L_{(1\ 1^2)}^{a,b}(s)=L_{(1\ 1^2)}^{a,b}(1-s)$, we have
\begin{align*}
&\tilde{L}_p(s,R) - \tilde{L}_p(1-s,R)\\
&= (p^{-s}-p^{s-1})L_{(1\ 1^2)}^{a,b}(s) -p^{T(1-s)-s}+p^{Ts+s-1}-(1-p^{-2s})\mathcal{C}(s)+(1-p^{2s-2})\mathcal{C}(1-s).
\end{align*}
Multiplying both sides of the \textbf{Key equation} in (\ref{eq:Key_equation}) by $p^{-s}$, we have 
\[
(p^{-s}-p^{s-1})L_{(1\ 1^2)}^{a,b}(s) -p^{T(1-s)-s}+p^{Ts+s-1}-(1-p^{-2s})\mathcal{C}(s)+(1-p^{2s-2})\mathcal{C}(1-s)=0.
\]
Hence, $\tilde{L}_p(s,R) = \tilde{L}_p(1-s,R)$ is satisfied.
\item
In the case where $\sigma_p(E)=(1\ 1\ 1)$, by Theorem \ref{thm:reformulation_for_split}, we have
\begin{align*}
    \tilde{L}_p(s,R)
    =&p^{T(1-s)}+(1-p^{-s})\Big(L_{(1\ 1^2)}^{a,b}(s) -p^{T(1-s)}+(1+p^{-s})\mathcal{C}(s)\Big)\\
    =&(1-p^{-s})L_{(1\ 1^2)}^{a,b}(s)+p^{T(1-s)-s}+(1-p^{-2s})\mathcal{C}(s).
\end{align*}
Evaluating its symmetry condition yields that
\begin{align*}
  &\tilde{L}_p(s,R) - \tilde{L}_p(1-s,R)\\
&= -(p^{-s}-p^{s-1})L_{(1\ 1^2)}^{a,b}(s) +p^{T(1-s)-s}-p^{Ts+s-1}+(1-p^{-2s})\mathcal{C}(s)-(1-p^{2s-2})\mathcal{C}(1-s).
\end{align*}
This equals $0$ by the above argument and hence $\tilde{L}_p(s,R)=\tilde{L}_p(1-s,R)$ is satisfied.
\end{itemize}

\vspace{0.5em}
\noindent \textbf{Part 3: Functional Equation for $\sigma_p(E)=(1\ 1^2)$ with odd $a(R_p)$.} \\
    In the case where $\sigma_p(E)=(1\ 1^2)$ and $a(R_p)$ is odd, we have $a(R_p)=2b(R_p)+1$ by Section \ref{sec:reformulation for ramlin}.
    By Theorem \ref{thm:reformulation_for_ramlin}, for $b=b(R_p)$, we have
        \begin{align*}
        \tilde{L}_p(s,R)=&p^{(3b+1)(1-s)}+p^{3b(1-s)+s}\sum_{i=0}^{b-1}p^{(3s-1)i}+(1-p^{-s})p^{(3b+1)s}\Bigg(p^{(3-6s)b}+\\
            &
            \sum_{d'=1}^{b}p^{(3-6s)(b-d')}\Big(\sum_{i=0}^{d'}p^{(2-3s)i}+p^{1-2s}\sum_{i=0}^{d'}p^{(2-3s)i}+p^{2-4s}\sum_{i=0}^{d'-1}p^{(2-3s)i}\Big)\Bigg).
        \end{align*}
 We consider the difference \begin{equation}\label{eq:intervene final}\tilde{L}_p(s,R)-p^{3b(1-s)+s}\sum_{i=0}^{b-1}p^{(3s-1)i}.\end{equation}
Applying telescoping cancellation to the double sums in \eqref{eq:intervene final} (as the factor $(1-p^{-s})$ induces a telescoping cancellation), we obtain 
\begin{align*}
\eqref{eq:intervene final}
    =&p^{(3b+1)s}\sum_{i=0}^b p^{(2-3s)i}+p^{3bs+1-s}\sum_{i=0}^{b}p^{(2-3s)i}+p^{(3b+1)(1-s)}\sum_{d'=0}^{b-1}p^{3d'(2s-1)}p^{(2-3s)d'}\\
    &-p^{(3b+1)(1-s)-(2-3s)}\Big(\sum_{d'=0}^{b}p^{(3d'-1)(2s-1)}+\sum_{d'=0}^{b-1}p^{3d'(2s-1)}+\sum_{d'=0}^{b-1}p^{(3d'+1)(2s-1)}\Big)\\
    =& p^{(3b+1)s}\sum_{i=0}^b p^{(2-3s)i}+p^{3bs+1-s}\sum_{i=0}^{b}p^{(2-3s)i}+p^{(3b+1)(1-s)}\sum_{i=0}^{b-1}p^{(3s-1)i}-p^{3b(1-s)}\sum_{i=0}^{3b}p^{(2s-1)i}.
\end{align*}
We then have
\begin{align*}
    \tilde{L}_p(s,R)=&\Big(p^{3b(1-s)+s}\sum_{i=0}^{b-1}p^{(3s-1)i}+p^{3bs+1-s}\sum_{i=0}^{b-1}p^{(2-3s)i}\Big)+\Big(p^{2b+1-s} +p^{2b+s}\Big)\\
    &+\Big(p^{(3b+1)s}\sum_{i=0}^{b-1}p^{(2-3s)i}+p^{(3b+1)(1-s)}\sum_{i=0}^{b-1}p^{(3s-1)i}\Big)-p^{3b(1-s)}\sum_{i=0}^{3b}p^{(2s-1)i}.
\end{align*}
Here, the last term has the following symmetry
\[p^{3b(1-s)}\sum_{i=0}^{3b}p^{(2s-1)i}=p^{3b(1-s)}\cdot p^{3b(2s-1)} \sum_{i=0}^{3b}p^{(1-2s)i}=p^{3bs}\sum_{i=0}^{3b}p^{(1-2s)i}.\]
This yields the functional equation for the case where $\sigma_p(E)=(1\ 1^2)$ with odd $a(R_p)$, completing the proof for all possible cases.\qedhere
\end{proof}        

\bibliographystyle{alpha}
\bibliography{References}

\end{document}